\documentclass{amsart}   	

\usepackage{amsmath}
\usepackage{amssymb}
\usepackage{amsfonts}
\usepackage{hyperref}
\usepackage{tikz-cd}
\usepackage{amscd}
\usepackage[all]{xy}
\usepackage{enumerate}

\def\Z{{\mathbb Z}} 
\def\Zl{{\mathbb{Z}_\ell}}
\def\Q{{\mathbb Q}} 
 
\def\C{{\mathbb C}}

\def\cC{{\mathcal C}}

\def\cG{{\mathcal{G}}}
\def\g{{\mathfrak{g}}}

\def\M{{\mathcal M}}

\def\p{{\mathfrak{p}}}

\def\U{{\mathcal U}}
\def\u{{\mathfrak{u}}}

\def\X{{\mathcal X}}
\def\Z{{\mathbb{Z}}}

\def\G{{\Gamma}}

\def\Ql{{\Q_\ell}}

\def\ab{\mathrm{ab}}
\def\an{\mathrm{an}}

\def\cts{\mathrm{cts}}

\def\prol{{(\ell)}}

\def\et{\mathrm{\acute{e}t}}
\def\diag{\mathrm{diag}}

\def\rel{{\mathrm{rel}(\ell)}}
 
\def\top{\mathrm{top}}
\def\orb{\mathrm{orb}}

\def\un{\mathrm{un}}

\newcommand{\Gr}{\operatorname{Gr}}
\newcommand{\Hom}{\operatorname{Hom}}

\newcommand{\Lie}{\operatorname{Lie}}

\newcommand\im{\operatorname{im}} 

\newcommand\Aut{\operatorname{Aut}}

\newcommand\Spec{\operatorname{Spec}}

\newcommand\Sp{\operatorname{Sp}} 

\newtheorem{theorem}{Theorem}[section]
\newtheorem{lemma}[theorem]{Lemma}
\newtheorem{proposition}[theorem]{Proposition}

\newtheorem{bigtheorem}{Theorem}
\newtheorem{bigcorollary}[bigtheorem]{Corollary}

\theoremstyle{definition}
\newtheorem{definition}[theorem]{Definition}

\theoremstyle{remark}
\newtheorem{remark}[theorem]{Remark}

\begin{document}

\title[Unordered Universal Curves in Characteristic $p$]
{On the universal curve with unordered marked points in positive characteristic}
\author{
Ma Luo \and Tatsunari Watanabe
}
\address{School of Mathematical Sciences,  Key Laboratory of MEA (Ministry of Education), Shanghai Key Laboratory of PMMP,  East China Normal University, Shanghai}
\email{mluo@math.ecnu.edu.cn}
\address{Mathematics Department, Embry-Riddle Aeronautical University, Prescott}
\email{watanabt@erau.edu}

\thanks{The first author is supported in part by National Natural Science Foundation of China (No. 12201217), by Shanghai Pilot Program for Basic Research, and by Science and Technology Commission of Shanghai Municipality (No. 22DZ2229014).}

\begin{abstract}
\smallskip
We study the relative pro-$\ell$ and continuous relative completions of the algebraic fundamental groups of universal curves over the moduli stack of curves with unordered marked points in positive characteristic. Using specialization and homotopy exact sequences, we compare the ordered and unordered settings and prove that the natural projection from the relative completion of the universal curve over the unordered moduli stack admits no section in positive characteristic. This yields a non-splitting result for the corresponding projection on algebraic fundamental groups. The present paper is a sequel to our earlier work in characteristic zero.
\end{abstract}
\maketitle

\setcounter{tocdepth}{1}
\tableofcontents
\section{Introduction}

Let $k$ be a field of characteristic $p >0$ and assume that $2g - 2 + n > 0$.
Denote by $\M_{g,n/k}$ the moduli stack of smooth proper curves of genus $g$
equipped with $n$ ordered, distinct marked points over $k$, and let
\[
\cC_{g,n/k} \longrightarrow \M_{g,n/k}
\]
be the universal curve.
The symmetric group $S_n$ acts freely (in the stack-theoretic sense) by permuting the marked
points, which induces quotient stacks
\[
\cC_{g,[n]/k} := [\cC_{g,n/k}/S_n],
\qquad
\M_{g,[n]/k} := [\M_{g,n/k}/S_n].
\]
We call the induced morphism
\[
\pi_{g,[n]/k}: \cC_{g,[n]/k} \longrightarrow \M_{g,[n]/k}
\]
the universal curve of genus $g$ with $n$ \emph{unordered} marked points over~$k$.

Fix an algebraic closure $\bar k$ of $k$.
\textcolor{black}{Let $\bar\xi$ be a geometric point of $\M_{g,[n]/\bar k}$, and
let $C$ denote the corresponding geometric fiber of
$\pi_{g,[n]/\bar k}$.}
Choose a geometric point $\bar{y}$ of $C$.
Then $\pi_{g,[n]/\bar k}$ induces the exact sequence of profinite groups
\begin{equation}\label{main homot seq in p}
1 \longrightarrow \pi_1(C,\bar{y})^{\prol}
\longrightarrow
\pi_1'\!\bigl(\cC_{g,[n]/\bar k},\bar{y}\bigr)
\longrightarrow
\pi_1\!\bigl(\M_{g,[n]/\bar k},\bar\xi\bigr)
\longrightarrow 1.
\end{equation}
Here $\pi_1(C,\bar{y})^{\prol}$ denotes the pro-$\ell$ completion of
$\pi_1(C,\bar{y})$,
and $\pi_1'\!\bigl(\cC_{g,[n]/\bar k},\bar{y}\bigr)$ is a canonical quotient of
$\pi_1\!\bigl(\cC_{g,[n]/\bar k},\bar{y}\bigr)$ obtained by the SGA~1 construction
(cf.~\cite[Exp.~XIII,~\S 4]{SGA1}).
\textcolor{black}{We briefly review this construction in
\S\ref{hom seq in p}.}

Passing to continuous relative completions with respect to the $\ell$-adic
monodromy representation
\[
\pi_1\!\bigl(\cC_{g,[n]/\bar k},\bar{y}\bigr)
\longrightarrow
\pi_1\!\bigl(\M_{g,[n]/\bar k},\bar\xi\bigr)
\longrightarrow
\Sp(H_\Ql),
\]
where $\ell \neq \operatorname{char}(k)$ and
\textcolor{black}{$H_\Ql := H^1_{\et}(C, \Ql(1))$,}
one obtains an exact sequence of pro-algebraic $\Q_\ell$--groups
\textcolor{black}{(see \S\ref{cont rel comp})}
\begin{equation}\label{main exact seq in p}
1 \longrightarrow F^\un_g
\longrightarrow
\cG_{\cC_{g,[n]}, \bar\xi}
\longrightarrow
\cG_{g,[n], \bar\xi}
\longrightarrow 1.
\end{equation}
\textcolor{black}{Here $F^\un_g$ is the $\ell$-adic unipotent completion of
$\pi_1(C, \bar y)$.}

\begin{bigtheorem}\label{thm:non-split-positive-char}
Suppose that $\operatorname{char}(k) =p > 0$, $\ell\not=p$ an odd prime, and
$g \geq 3$. Then the exact sequence \eqref{main exact seq in p} does not admit
any section.
\end{bigtheorem}

\begin{bigcorollary}\label{cor:no-section}
If $g \ge 3$, the canonical surjection
\[
\pi_1\!\bigl(\cC_{g,[n]/k},\bar y\bigr)
\longrightarrow
\pi_1\!\bigl(\M_{g,[n]/k},\bar\xi\bigr)
\]
admits no continuous section. In particular, the universal curve
\[
\cC_{g,[n]/k} \longrightarrow \M_{g,[n]/k}
\]
has no section.
\end{bigcorollary}

\vspace{0.5em}

This paper is a sequel to~\cite{LW25}, where the corresponding nonsplitting
results for universal curves with unordered marked points were established in
characteristic~$0$. Our aim here is to extend those results to positive
characteristic. Although the overall strategy follows the same general
framework as in the characteristic~$0$ case, carrying out this extension
requires substantial additional work.
\textcolor{black}{In characteristic~$0$, the essential geometric information
needed for the argument is encoded in the weight filtration on relative and
continuous relative completions.}
In positive characteristic, while continuous relative completion remains
available for algebraic fundamental groups, transferring the corresponding
structural properties across characteristics is far from automatic.
\textcolor{black}{To address this, we make systematic use of relative pro-$\ell$
completion and continuous relative completion, together with comparison
isomorphisms between characteristic~$0$ and characteristic~$p$ established
in~\cite{wat_rational_points_inp}, which allow us to transport the necessary
monodromy and homotopy-theoretic information into characteristic~$p$.}
These tools play a central role throughout the paper and are indispensable for
the arguments that follow. Familiarity with the methods developed
in~\cite{LW25} will therefore be beneficial to the reader.
\section{Moduli of Curves with unordered marked points}\label{modui of curves}

Let $k$ be a field of characteristic $p>0$, and fix an algebraic closure
$\bar{k} \subset \Omega$, where $\Omega$ is a chosen algebraically closed field
containing~$k$.
\textcolor{black}{Assume throughout that $g,n \ge 0$ and $2g-2+n>0$.}

Over $\Spec \Z$, let $\M_{g,n/\Z}$ denote the moduli stack of smooth, proper
curves of genus~$g$ with $n$ ordered distinct marked points.
Its base change
\[
  \M_{g,n/k} := \M_{g,n/\Z} \times_{\Spec \Z} \Spec k
\]
is a smooth Deligne--Mumford stack over~$k$.

The symmetric group $S_n$ defines a constant finite \'etale group scheme over $\Spec \Z$.
By base change,
\[
  S_{n/k} := S_n \times_{\Spec \Z} \Spec k
\]
is a constant finite \'etale group scheme over $k$ acting on $\M_{g,n/k}$ by permuting the marked points.
Since $\M_{g,n/k}$ is a smooth Deligne--Mumford stack, the quotient stack
\[
  \M_{g,[n]/k} := [\,\M_{g,n/k} / S_{n/k}\,]
\]
is also a smooth Deligne--Mumford stack over~$k$.
It parametrizes smooth, proper genus~$g$ curves with $n$
\emph{unordered} marked points.

The universal curve over $\M_{g,n/\Z}$ base changes to the universal curve over
$k$:
\[
  \pi_{g,n/k} : \cC_{g,n/k} \longrightarrow \M_{g,n/k}.
\]
The $S_{n/k}$-action on $\M_{g,n/k}$ lifts to $\cC_{g,n/k}$, and the resulting
quotient stack
\[
  \cC_{g,[n]/k} := [\,\cC_{g,n/k} / S_{n/k}\,]
\]
is a Deligne--Mumford stack with induced morphism
\[
  \pi_{g,[n]/k} : \cC_{g,[n]/k} \longrightarrow \M_{g,[n]/k},
\]
which we refer to as the universal curve of genus~$g$ with $n$
\emph{unordered} marked points over~$k$.
Both projections
\[
  \M_{g,n/k} \longrightarrow \M_{g,[n]/k}
  \qquad\text{and}\qquad
  \cC_{g,n/k} \longrightarrow \cC_{g,[n]/k}
\]
are finite étale $S_{n/k}$-torsors.

\medskip

Finally, for any connected Deligne--Mumford stack $\X$ over a field $L$, we write
$\pi_1(\X_{/L})$ for its (profinite) algebraic fundamental group in the sense of
SGA~1. When working over $\C$, we instead use the notation $\pi_1^\top$ and
$\pi_1^{\orb}$ for the topological and orbifold fundamental groups,
respectively.
\subsection{Homotopy exact sequences in positive characteristic}\label{hom seq in p}

In this subsection, we briefly recall the SGA~1 construction of the homotopy
exact sequence for a proper smooth morphism with geometrically connected
fibers, and then apply it to the ordered and unordered universal curves
(cf.~\cite[SGA~1, Exp.~XIII, §4]{SGA1}; see also
\cite[Prop.~4.2]{wat_rational_points_inp}).

Let $L$ be a field, and let $T$ be a connected, locally noetherian
Deligne--Mumford stack over $L$.
Let $f : C \longrightarrow T$
be a proper smooth morphism whose geometric fibers are connected curves of
genus $g\ge 2$.
Fix geometric points
  $\bar\zeta$ in $T$ and 
  $\bar{a}$ in the fiber $C_{\bar\zeta}$ of $f$ over $\bar\zeta$.
Let $\ell$ be a prime number different from
$\operatorname{char}(k)$.
Following
\cite[SGA~1, Exp.~XIII, §4.3--4.4]{SGA1}, set
\[
  K := \ker\!\bigl(\pi_1(C,\bar{a}) \longrightarrow \pi_1(T,\bar\zeta)\bigr),
  \qquad
  K^\prol:= \text{maximal pro-}\ell\text{ quotient of } K,
\]
and let
\[
  N := \ker\!\bigl(K \longrightarrow K^\prol\bigr).
\]
Then $N$ is a distinguished closed normal subgroup of $\pi_1(C,\bar{a})$, and we
define
\[
  \pi_1'(C,\bar{a}) := \pi_1(C,\bar{a}) / N.
\]
We also write $\pi_1(C_{\bar\zeta},\bar{a})^\prol$ for the maximal
pro-$\ell$ quotient of $\pi_1(C_{\bar\zeta},\bar{a})$.
With this notation,
\cite[Prop.~4.2]{wat_rational_points_inp} yields the exact sequence
\begin{equation}\label{eq:prime-to-L-exact}
  1 \longrightarrow \pi_1(C_{\bar\zeta},\bar{a})^\prol
    \longrightarrow \pi_1'(C,\bar{a})
    \longrightarrow \pi_1(T,\bar\zeta)
    \longrightarrow 1.
\end{equation}

\medskip
Assume $g \ge 2$, and let $L$ be a field of characteristic $p \ge 0$.
Consider the ordered and unordered universal curves
\[
\pi_{g,n/L} : \cC_{g,n/L} \longrightarrow \M_{g,n/L},
\qquad
\pi_{g,[n]/L} : \cC_{g,[n]/L} \longrightarrow \M_{g,[n]/L}.
\]
Fix a geometric point $\bar{\zeta}$ of $\M_{g,n/L}$, which we also view as a
geometric point of $\M_{g,[n]/L}$ via the finite \'etale projection
$\M_{g,n/L} \to \M_{g,[n]/L}$. Let $C_{\bar{\zeta}}$ denote the corresponding
geometric fiber, and choose a geometric point $\bar{a} \in C_{\bar{\zeta}}$.
Via the natural morphisms
\[
C_{\bar{\zeta}} \longrightarrow \cC_{g,n/L}
\longrightarrow \cC_{g,[n]/L},
\]
we regard $\bar{a}$ as a geometric point of both $\cC_{g,n/L}$ and
$\cC_{g,[n]/L}$.

Applying the SGA~1 construction \eqref{eq:prime-to-L-exact} to
$\pi_{g,n/L}$ yields an exact sequence
\begin{equation}\label{eq:ordered-prol-exact}
  1 \longrightarrow \pi_1(C_{\bar{\zeta}},\bar{a})^{(\ell)}
    \longrightarrow \pi_1'\bigl(\cC_{g,n/L},\bar{a}\bigr)
    \longrightarrow \pi_1\bigl(\M_{g,n/L},\bar{\zeta}\bigr)
    \longrightarrow 1.
\end{equation}
Likewise, applying the same construction to the unordered universal curve
$\pi_{g,[n]/L}$ gives an exact sequence
\begin{equation}\label{eq:unordered-prol-exact}
  1 \longrightarrow \pi_1(C_{\bar{\zeta}},\bar{a})^{(\ell)}
    \longrightarrow \pi_1'\bigl(\cC_{g,[n]/L},\bar{a}\bigr)
    \longrightarrow \pi_1\bigl(\M_{g,[n]/L},\bar{\zeta}\bigr)
    \longrightarrow 1.
\end{equation}

Passing from ordered to unordered moduli is induced by the finite \'etale
$S_n$-torsors
$\M_{g,n/L}\to \M_{g,[n]/L}$ and $\cC_{g,n/L}\to \cC_{g,[n]/L}$.
These morphisms induce maps between the exact sequences
\eqref{eq:ordered-prol-exact} and \eqref{eq:unordered-prol-exact}, summarized in
the following commutative diagram:
\begin{equation}\label{diag:ordered-unordered}
\begin{tikzcd}
1 \ar[r] &
\pi_1(C_{\bar{\zeta}},\bar{a})^{(\ell)} \ar[r] \ar[d, equal] &
\pi_1'\bigl(\cC_{g,n/L},\bar{a}\bigr) \ar[r] \ar[d] &
\pi_1\bigl(\M_{g,n/L},\bar{\zeta}\bigr) \ar[r] \ar[d] &
1 \\
1 \ar[r] &
\pi_1(C_{\bar{\zeta}},\bar{a})^{(\ell)} \ar[r] &
\pi_1'\bigl(\cC_{g,[n]/L},\bar{a}\bigr) \ar[r] &
\pi_1\bigl(\M_{g,[n]/L},\bar{\zeta}\bigr) \ar[r] &
1 .
\end{tikzcd}
\end{equation}
The vertical arrows are induced by the finite \'etale quotient maps associated
with the $S_n$-action. In particular, the right-hand square of
Diagram~\eqref{diag:ordered-unordered} is a pullback square in the category of
profinite groups.
\section{Configuration spaces and their fundamental groups}

Let $L$ be an algebraically closed field and let $C$ be a proper smooth
connected curve over $L$. For a positive integer $n$, the ordered
configuration space of $n$ points on $C$ is the open subscheme
\[
  F(C,n) := C^n \setminus \Delta,
\]
where $\Delta \subset C^n$ denotes the fat diagonal
\[
  \Delta := \{(x_1,\dots,x_n) \in C^n \mid x_i = x_j \text{ for some } i \neq j\}.
\]
Thus $F(C,n)$ parametrizes $n$-tuples of pairwise distinct points of~$C$.

The symmetric group $S_n$ acts freely on $F(C,n)$ by permuting the ordered
points, and the quotient
\[
  F(C,[n]) := F(C,n)/S_n
\]
is the unordered configuration space, parametrizing unordered $n$-tuples of
distinct points of $C$. The projection
\[
  F(C,n) \longrightarrow F(C,[n])
\]
is therefore a finite \'etale $S_n$-torsor.

In particular, for a geometric point $\bar a$ of $F(C, n)$ (also viewed as a
geometric point in $F(C,[n])$), one obtains an exact sequence of \'etale
fundamental groups
\[
  1 \longrightarrow \pi_1\bigl(F(C,n),\bar a\bigr)
    \longrightarrow \pi_1\bigl(F(C,[n]),\bar a\bigr)
    \longrightarrow S_n
    \longrightarrow 1.
\]
\subsection{The pro-$\ell$ completion of $\pi_1(F(C, n))$}\label{pro-l comp of F(C,n)}

Suppose that $k$ is a field of characteristic $p>0$, and let $\bar{k}$ be an
algebraic closure of $k$. Let $C$ be a connected, proper, smooth curve of genus
$g \ge 3$ over $\bar{k}$. Choose a complete DVR $R$ of characteristic $0$ with
residue field $\bar{k}$ and a proper, smooth morphism
\[
X \longrightarrow \Spec R
\]
whose special fiber is isomorphic to $C$.
\textcolor{black}{(For example, one may take $R = W(\bar{k})$, the ring of Witt
vectors, which is a complete DVR of mixed characteristic $(0,p)$.)}

Let $K$ be the fraction field of $R$ and $\bar{K}$ an algebraic closure of $K$.
For an integer $n \ge 1$, let $F(X,n)$ denote the ordered configuration space of
$X$ with $n$ labeled points. The Fulton--MacPherson compactification~\cite{FM} of
$F(X,n)$ is proper and smooth over $R$, and $F(X,n)$ is the complement of a
relative normal crossings divisor.

Let $\bar{\eta} : \Spec \bar{K} \to \Spec R$ and
$\bar{\xi} : \Spec \bar{k} \to \Spec R$ be the geometric points corresponding to
$\bar{K}$ and $\bar{k}$, respectively, and note that
\textcolor{black}{$X_{\bar{\xi}} \cong C$.}
Fix an embedding
$\iota : \bar{K} \hookrightarrow \C$,
and denote the associated complex curve of $X_{\bar{\eta}}$ by
\[
C^\an := X_{\bar{\eta}} \otimes_{\bar{K},\iota} \C.
\]
\textcolor{black}{Whenever we use an analytic description, we implicitly pass
through such a choice of~$\iota$.}
Then
\[
F(X_{\bar{\eta}},n) \;\cong\; F(X,n)_{\bar{K}}
\quad\text{and}\quad
F(C,n) \;\cong\; F(X,n)_{\bar{k}}.
\]
Let $\bar{x}$ and $\bar{y}$ be geometric points of $F(X_{\bar{\eta}},n)$ and
$F(C,n)$, respectively. Via the natural morphisms
$F(X_{\bar{\eta}},n) \longrightarrow F(X,n)$ and $F(C,n) \longrightarrow F(X,n)$,
we regard $\bar{x}$ and $\bar{y}$ as geometric points of $F(X,n)$. Fixing an
isomorphism
\[
\lambda :
\pi_1\bigl(F(X,n),\bar{x}\bigr)
\xrightarrow{\sim}
\pi_1\bigl(F(X,n),\bar{y}\bigr),
\]
we obtain a specialization homomorphism
\[
\mathrm{sp}_{\ell} :
\pi_1\bigl(F(X_{\bar{\eta}},n),\bar{x}\bigr)^{\prol}
\;\longrightarrow\;
\pi_1\bigl(F(C,n),\bar{y}\bigr)^{\prol},
\]
for each prime $\ell \ne p$. By
\cite[Expos\'e~XIII, Thm.~2.10, Cor.~2.9]{SGA1}
\textcolor{black}{and the fact that the Fulton--MacPherson compactification is
smooth over $R$ and that $F(X,n)$ is the complement of a relative normal
crossings divisor,}
this specialization map $\mathrm{sp}_{\ell}$ is an isomorphism.

Denote $\pi_1^\top(F(C^\an,n))$ by $P_{g,n}$. Then
\textcolor{black}{$\pi_1\bigl(F(X_{\bar{\eta}},n),\bar{x}\bigr)$ is naturally
isomorphic to the profinite completion $\widehat{P}_{g,n}$,}
and hence we obtain an isomorphism
\begin{equation}\label{pure braid in p}
\pi_1\bigl(F(C,n),\bar{y}\bigr)^{\prol}
\;\cong\;
P_{g,n}^{\prol}.
\end{equation}

\subsection{The $\ell$-adic unipotent completion}

Let $\ell\neq p$ be a prime.  Fix a geometric point $\bar{y}$ of $F(C,n)$ (also view it as a geometric point of $F(C, [n])$) and
for simplicity write 
\[
F_{g,n} := \pi_1\bigl(F(C,n),\bar{y}\bigr),
\qquad
F_{g,[n]} := \pi_1\bigl(F(C, [n]),\bar{y}\bigr).
\]
 Note $F_{g,1}=\pi_1(C, \bar y)$.  Denote by
$F_{g,n}^{\prol}$ and $F_{g,[n]}^{\prol}$ their pro-$\ell$ completions:
\[
F_{g,n}^{\prol} := \pi_1\bigl(F(C,n),\bar{y}\bigr)^{(\ell)},
\qquad
F_{g,[n]}^{\prol} := \pi_1\bigl(F(C, [n]),\bar{y}\bigr)^{(\ell)}.
\]



We now pass to the $\ell$-adic unipotent completions. Following
\cite[Appendix A]{hain-matsumoto}, let
$F_{g,n}^{\un}$ and $F_{g,[n]}^{\un}$
denote the $\ell$-adic unipotent completions of 
$F_{g,n}$ and $F_{g, [n]}$ over $\Q_\ell$,
respectively. These are pro-unipotent algebraic $\Ql$-groups
equipped with Zariski-dense continuous homomorphisms
\[
F_{g,n} \longrightarrow F_{g,n}^{\un}(\Q_\ell),
\qquad
F_{g, [n]} \longrightarrow F_{g,[n]}^{\un}(\Q_\ell).
\]
Since compact subgroups of pro-unipotent $\Q_\ell$-groups are
pro-$\ell$, these maps factor through the pro-$\ell$ completions. Thus the induced maps
\[
F_{g,n}^{\prol} \longrightarrow F_{g,n}^{\un}(\Q_\ell),
\qquad
F_{g,[n]}^{\prol} \longrightarrow F_{g,[n]}^{\un}(\Q_\ell)
\]
are again Zariski-dense.
We define their Lie algebras
\[
\p_{g,n} := \Lie\bigl(F_{g,n}^{\un}\bigr),
\qquad 
\p_{g,[n]} := \Lie\bigl(F_{g,[n]}^{\un}\bigr),
\]
which are pronilpotent Lie algebras over $\Q_\ell$. When $n=1$, we write $F_g^{\un}$ and $\p_g$ for
$F_{g,1}^{\un}$ and $\p_{g,1}$.

For a group $\G$ and a field $L$, set $H_1(\G, L):= \G^\ab\otimes_\Z L$.
\begin{lemma}\label{lem:H1-duality} Let $L$ be an algebraically closed field of characteristic $p \ge 0$, and let
$\ell$ be a prime with $\ell \neq p$. Let $C$ be a proper smooth curve over $L$.
There is a canonical isomorphism of $\Q_\ell$-vector spaces
\[
H^1_{\et}(C,\Q_\ell(1))
\;\cong\;
H_1\bigl(\pi_1(C)^{\prol}, \Q_\ell\bigr),
\]
compatible with the natural symplectic form on the left-hand side.
\end{lemma}

\begin{proof}
Since $C$ is proper and smooth over $L$, the
comparison between \'etale cohomology of $C$ and continuous cohomology of its
profinite fundamental group yields
\[
H^1_{\et}(C,\Q_\ell(1))
\;\cong\;
H^1_{\cts}\bigl(\pi_1(C),\Q_\ell(1)\bigr)
\;\cong\;
\Hom_{\cts}\!\bigl(\pi_1(C),\Q_\ell(1)\bigr).
\]
Since any continuous map $\pi_1(C)\to\Q_\ell(1)$ factors through the 
pro-$\ell$ abelianization of $\pi_1(C)$, we obtain canonical isomorphisms
\begin{align*}
\Hom_{\cts}\!\bigl(\pi_1(C),\Q_\ell(1)\bigr)
&\cong
  \Hom_{\cts}\!\bigl(\pi_1(C)^{\prol},\Q_\ell(1)\bigr)\\
&\cong
  \Hom_{\Q_\ell}\!\bigl(H_1(\pi_1(C)^{\prol},\Q_\ell),\Q_\ell(1)\bigr).
\end{align*}
Since $H_1(\pi_1(C)^{\prol},\Q_\ell)\cong \Q_\ell^{2g}$ is finite-dimensional, the last
space is canonically isomorphic to
$
H_1(\pi_1(C)^{\prol},\Q_\ell)^\vee \otimes \Q_\ell(1).
$
On the other hand, there is a perfect alternating
pairing
\[
H^1_{\et}(C,\Q_\ell(1)) \otimes H^1_{\et}(C,\Q_\ell(1))
   \longrightarrow H^2_{\et}(C,\Q_\ell(2))
   \cong \Q_\ell(1),
\]
which induces a canonical isomorphism
$
H^1_{\et}(C,\Q_\ell(1))
   \xrightarrow{\sim}
   H^1_{\et}(C,\Q_\ell(1))^{\vee}(1).
$
Combining this with the identification above, we compute
\begin{align*}
H^1_{\et}(C,\Q_\ell(1))
&\cong
\bigl(H_1(\pi_1(C)^{\prol},\Q_\ell)^{\vee}\otimes\Q_\ell(1)\bigr)^{\vee}
   \otimes\Q_\ell(1)\\
&\cong
H_1(\pi_1(C)^{\prol},\Q_\ell)\otimes\Q_\ell(-1)\otimes\Q_\ell(1)\\
&\cong
H_1(\pi_1(C)^{\prol},\Q_\ell).
\end{align*}
\end{proof}

\subsection{Abelianizations of unipotent completions of configuration spaces}
We study the abelianizations $H_1(\p_{g,n})$ and
$H_1(\p_{g,[n]})$ \textcolor{black}{via} the comparison with the characteristic zero case.
Set $H_{\Q_\ell} := H^1_{\et}(C,\Q_\ell(1))$.  
The $S_n$-action on $F(C,n)$ by permuting points induces an action on homology.
The following proposition confirms that the familiar characteristic-zero
behavior holds equally in positive characteristic.
\begin{proposition}\label{prop:H1-pgn} For $g\ge 0$ and $n\ge 1$, 
there are natural isomorphisms
\[
H_1(\p_{g,n})
\;\cong\;
H_1\bigl(F_{g,n}^{\prol},\Q_\ell\bigr)
\;\cong\;
H_{\Q_\ell}^{\oplus n}.
\]
\end{proposition}
\begin{proof}
Since compact subgroups of the $\Ql$-rational points of a unipotent $\Ql$-group are pro-$\ell$ groups, it follows that for a profinite group $\G$, there is a canonical isomorphism $\G^\un\cong \G^{\prol\un}$
(cf.~\cite[Appendix A]{hain-matsumoto}). Furthermore, if $H_1(\G, \Ql): = \G^\ab\otimes_\Z\Ql$ is finite dimensional,  we have a canonical isomorphism $H_1(\Lie(\G^\un))\cong H_1(\G, \Ql)$. For $\G =F_{g,n}$, recall that there is a natural isomorphism $F^\prol_{g,n} \cong P_{g,n}^\prol$, and hence we have $H_1(F^\prol_{g,n}, \Ql)\cong H_1(P^\prol_{g,n}, \Ql)$. 
It is a classical computation (see, for example, \cite[Prop.~2.1]{hain_infpretor}) that
\[
H_1\bigl(P_{g,n},\Z\bigr)
\;\cong\;H_1\bigl(F(C^\an, n),\Z\bigr)\;\cong\;
H_1(C^\an,\Z)^{\oplus n}.
\]
Since $P_{g,n}$ is finitely generated, tensoring with $\Q_\ell$, we obtain
\[
H_1\bigl(P_{g,n},\Q_\ell\bigr)
\;\cong\;H_1\bigl(P_{g,n}^\prol,\Q_\ell\bigr)\;\cong\;
H_1(C^\an,\Q_\ell)^{\oplus n}.
\]
Finally, since $\pi_1^\top(C^\an)¢$ is finitely generated, combining the specialization isomorphism
$\pi_1(C)^{\prol}\cong \pi_1^\top(C^\an)^{(\ell)}$
with Lemma~\ref{lem:H1-duality}, we have natural isomorphisms
\[
H_1(C^\an,\Q_\ell)
\;\cong\;
H_1\bigl(\pi_1^\top(C^\an)^{\prol},\Q_\ell\bigr)
\;\cong\;H_1\bigl(\pi_1(C)^{\prol},\Q_\ell\bigr)\;\cong\;
H_{\Q_\ell},
\]
and therefore
\[
H_1\bigl(F_{g,n}^{\prol},\Q_\ell\bigr)\cong H_1\bigl(P_{g,n}^{\prol},\Q_\ell\bigr)
\;\cong\;
H_{\Q_\ell}^{\oplus n}.
\]
Therefore, $H_1(F^\prol_{g,n}, \Ql)$ is finite dimensional, and so there are natural isomorphisms
$$
H_1(\p_{g,n})\cong H_1\bigl(\Lie(F^{\prol\un}_{g,n})\bigr)\cong H_1\bigl(F^\prol_{g,n}, \Ql\bigr) \cong H_{\Q_\ell}^{\oplus n}.
$$
\end{proof}
The previous result computes the first homology in the ordered case.  We now
pass to the unordered configuration space by taking $S_n$-coinvariants.
\begin{proposition}\label{prop:H1-pgn-quot}
For $g\ge 0$ and $n\ge 1$, there are natural isomorphisms
\[
H_1(\p_{g,[n]})
\;\cong\;
H_1\bigl(F_{g,[n]}^{\prol},\Q_\ell\bigr)
\;\cong\;
\bigl(H_{\Q_\ell}^{\oplus n}\bigr)_{S_n}
\;\cong\;
H_{\Q_\ell}.
\]
\end{proposition}
\begin{proof}
As in the ordered case, unipotent completion induces an isomorphism on
$\Q_\ell$-rationalized abelianizations, so we have a canonical identification
\[
H_1(\p_{g,[n]})
\;\cong\;
H_1\bigl(F_{g,[n]}^{\prol},\Q_\ell\bigr).
\]

The finite \'etale Galois cover
\[
F(C,n)\longrightarrow F(C, [n])
\]
with group $S_n$ induces an exact sequence of algebraic fundamental groups
\[
1 \longrightarrow F_{g,n}
  \longrightarrow F_{g,[n]}
  \longrightarrow S_n
  \longrightarrow 1.
\]
The Hochschild--Serre spectral sequence for group homology with coefficients
in $\Q_\ell$ has
\[
E^2_{p,q}
=
H_p\bigl(S_n,\,H_q(F_{g,n},\Q_\ell)\bigr)
\;\Longrightarrow\;
H_{p+q}\bigl(F_{g,[n]},\Q_\ell\bigr).
\]
Since $S_n$ is finite and $|S_n|$ is invertible in $\Q_\ell$, one has
$H_p(S_n,V)=0$ for all $p>0$ and all $\Q_\ell[S_n]$-modules $V$.  In degree~1
this yields a canonical isomorphism
\[
H_1\bigl(F_{g,[n]},\Q_\ell\bigr)
\;\cong\;
H_0\bigl(S_n,\,H_1(F_{g,n},\Q_\ell)\bigr)
=
H_1(F_{g,n},\Q_\ell)_{S_n}.
\]

Indeed, for a profinite group $G$, the abelianization commutes with passage to the
maximal pro-$\ell$ quotient, and tensoring with $\Q_\ell$ kills all prime-to-$\ell$
components. Consequently,
\[
G^{\ab}\otimes \Q_\ell \;\cong\; (G^\ab)^{\prol}\otimes \Q_\ell\;\cong\;(G^{\prol})^{\ab}\otimes \Q_\ell,
\]
and the natural homomorphism $G \to G^{\prol}$ induces an isomorphism on
$H_1(-,\Q_\ell)$.  
Thus
\[
H_1\bigl(F_{g,[n]}^{\prol},\Q_\ell\bigr)
\;\cong\;
H_1\bigl(F_{g,n}^{\prol},\Q_\ell\bigr)_{S_n}.
\]
By the computation in the ordered case (Proposition~\ref{prop:H1-pgn}),
\[
H_1\bigl(F_{g,n}^{\prol},\Q_\ell\bigr)
\;\cong\;
H_{\Q_\ell}^{\oplus n},
\]
where $S_n$ acts by permuting the $n$
summands.  Hence
\[
H_1\bigl(F_{g,[n]}^{\prol},\Q_\ell\bigr)
\;\cong\;
\bigl(H_{\Q_\ell}^{\oplus n}\bigr)_{S_n}.
\]

Finally, the $S_n$-coinvariants of $H_{\Q_\ell}^{\oplus n}$ are canonically
isomorphic to $H_{\Q_\ell}$.
Combining the above identifications gives the natural isomorphisms
\[
H_1(\p_{g,[n]})
\;\cong\;
H_1\bigl(F_{g,[n]}^{\prol},\Q_\ell\bigr)
\;\cong\;
\bigl(H_{\Q_\ell}^{\oplus n}\bigr)_{S_n}
\;\cong\;
H_{\Q_\ell},
\]
which proves the proposition.
\end{proof}


\section{Relative Pro-$\ell$ Completion}
A key ingredient in this work is the study of the algebraic fundamental group
$\pi_1(\M_{g/\bar{k}})$ for a field $k$ of characteristic $p>0$.
A natural first idea is to compare it with the mapping class group
\[
\G_g \;\cong\; \pi_1^{\orb}(\M_{g/\C})
\]
via pro-$\ell$ completion.
However, when $g \ge 3$, the mapping class group $\G_g$ is perfect, and hence its
pro-$\ell$ completion is trivial.
As a result, the usual pro-$\ell$ completion loses all information and cannot be
used to transfer the structure of $\pi_1(\M_{g/\C})$ to the study of
$\pi_1(\M_{g/\bar{k}})$.

To overcome this, one instead considers the \emph{relative} pro-$\ell$
completion. This construction retains essential information about the mapping class group by
taking into account the natural $\ell$-adic monodromy representation
$
\rho : \G_g \longrightarrow \Sp_{2g}(\Z_\ell)
$
and completing $\G_g$ in a manner that is compatible with~$\rho$.  In this way,
the resulting relative pro-$\ell$ completion preserves the $\ell$-adic structure
coming from the geometry of $\M_g$.  As shown in~\cite{HainMatsumoto2009}, the resulting relative
pro-$\ell$ completion is nontrivial and
captures the geometric structure of $\M_g$ much more effectively than the
ordinary pro-$\ell$ completion.
For instance, when marked points are present, the mapping class group embeds
into its relative pro-$\ell$ completion (\cite[Prop.~3.1]{HainMatsumoto2009}).

Because the second author applied this approach in the positive characteristic
setting in~\cite{wat_rational_points_inp}, we briefly recall below the
definition and the basic properties of relative pro-$\ell$ completion.  Full
details and extensive background can be found in~\cite{HainMatsumoto2009}.

\begin{definition}[Relative pro-$\ell$ completion]
Let $\Gamma$ be a discrete group  or a profinite group and let
\[
\rho : \Gamma \longrightarrow P
\]
be a continuous homomorphism into a profinite group $P$ whose image is dense.
The \emph{relative pro-$\ell$ completion} of $\Gamma$ with respect to~$\rho$ is
a profinite group $\Gamma^{\rel}$ equipped with a continuous homomorphism
$\rho^\prol : \Gamma \to \Gamma^{\rel}$ satisfying the following properties:
\begin{enumerate}
  \item there is an exact sequence of profinite groups
  \[
  1 \longrightarrow K \longrightarrow \Gamma^{\rel}
    \xrightarrow{\;\tilde\rho\;} P \longrightarrow 1,
  \]
  where $K$ is a pro-$\ell$ group and the map
$\tilde\rho : \Gamma^{\rel} \to P$ is induced by~$\rho$ such that $\tilde\rho \circ \rho^\prol = \rho$;
  \item the image $\rho^\prol(\Gamma)$ is dense in $\Gamma^{\rel}$;
  \item $\Gamma^{\rel}$ is universal with these properties: if $G$ is a
  profinite group equipped with a surjection $G \to P$ whose kernel is pro-$\ell$
  and if $\varphi : \Gamma \to G$ is a continuous homomorphism lifting $\rho$,
  then there exists a unique continuous homomorphism
  $\tilde\varphi:\Gamma^{\rel} \to G$ such that $\varphi = \tilde\varphi\circ \rho^\prol$.
\end{enumerate}
\end{definition}

\subsection{The relative pro-$\ell$ completion of the mapping class group}

Let $L$ be a field of characteristic $p \ge 0$, and fix an algebraic closure
$\bar L$ of $L$.  Let
$\pi := \pi_{g,n/\bar L} : \cC_{g,n/\bar L} \longrightarrow \M_{g,n/\bar L}$
denote the universal curve.  Fix a prime $\ell \neq p$, and consider the
$\ell$-adic local system
\[
\mathbb H_{\Z_\ell} := R^1\pi_* \Z_\ell(1)
\]
on $\M_{g,n/\bar L}$.
For a geometric point $\bar\zeta$ of $\M_{g,n/\bar L}$, let $C_{\bar\zeta}$ be the
corresponding geometric fiber of $\pi$, and set
$H_{\Z_\ell} := \mathbb H_{\Z_\ell,\bar\zeta}
  = H^1_{\et}(C_{\bar\zeta}, \Z_\ell(1))$.
The natural cup product pairing endows $H_{\Z_\ell}$ with a symplectic
structure, yielding a monodromy representation
\[
\rho_{\Z_\ell,\bar\zeta} :
\pi_1(\M_{g,n/\bar L}, \bar\zeta)
\longrightarrow
\Sp(H_{\Z_\ell}).
\]

When $\operatorname{char}(L)=p>0$ and $\ell\neq p$, the second author proved in
\cite[Thm.~5.6]{wat_rational_points_inp} that the relative pro-$\ell$ completion
of $\pi_1(\M_{g,n/\bar L}, \bar\zeta)$ with respect to
$\rho_{\Z_\ell,\bar\zeta}$ is naturally isomorphic to $\G_{g,n}^{\rel}$, the
relative pro-$\ell$ completion of the mapping class group $\G_{g,n}$ with
respect to its standard symplectic representation
\[
\G_{g,n} \longrightarrow \Sp\bigl(H_1(\Sigma_g,\Z_\ell)\bigr),
\]
where $\Sigma_{g}$ is a compact oriented surface of genus $g$. 

\subsection{Specialization for the universal curve in the ordered case} \label{subsec:homotopy-rel-prol}
Let $k$ be a field of characteristic $p>0$ with algebraic closure $\bar k$.
Let $R := W(\bar k)$ be the ring of Witt vectors of $\bar k$, let
$ K := \mathrm{Frac}(R)$ be its fraction field, and fix an algebraic closure
$\bar K$ of $K$.
Let \(\bar{\xi}\) and \(\bar{\eta}\) be geometric points of
\(\mathcal{M}_{g,n/\bar{k}}\) and \(\mathcal{M}_{g,n/\bar{K}}\), respectively.
Via the canonical morphisms
\(\mathcal{M}_{g,n/\bar{k}} \to \mathcal{M}_{g,n/R}\) and
\(\mathcal{M}_{g,n/\bar{K}} \to \mathcal{M}_{g,n/R}\),
we also regard \(\bar{\xi}\) and \(\bar{\eta}\) as geometric points of
\(\mathcal{M}_{g,n/R}\).
Let \(\pi_{g,n/R} : \mathcal{C}_{g,n/R} \to \mathcal{M}_{g,n/R}\) be the universal
curve, and denote by \(C_{\bar{\xi}}\) and \(C_{\bar{\eta}}\) its fibers over
\(\bar{\xi}\) and \(\bar{\eta}\), respectively. Fix geometric points
\(\bar{y} \in C_{\bar{\xi}}\) and \(\bar{x} \in C_{\bar{\eta}}\).
\begin{proposition}\label{prop:specialization-rel-prol}
For $g\ge 2$ and $n \ge 0$, there is a commutative diagram: 
$$
\xymatrix@R=9pt@C=9pt{
1 \ar[r] &
\pi_1(C_{\bar{\xi}}, \bar{y})^{\prol} \ar[r]  &
\pi_1(\mathcal{C}_{g,n/\bar k}, \bar{y})^{\rel} \ar[r]  &
\pi_1(\mathcal{M}_{g,n/\bar k}, \bar{\xi})^{\rel} \ar[r]  &
1 \\
1 \ar[r] &
\pi_1(C_{\bar{\eta}}, \bar{x})^{\prol} \ar[r] \ar[u]_\cong&
\pi_1(\mathcal{C}_{g,n/\bar{K}}, \bar{x})^{\rel} \ar[r] \ar[u]_{\cong} &
\pi_1(\mathcal{M}_{g,n/\bar{K}}, \bar{\eta})^{\rel} \ar[r] \ar[u]_{\cong} &
1 ,
}
$$
where the rows are exact and all vertical maps are isomorphisms.
\end{proposition}
\begin{proof}
Fix an isomorphism
\[
\lambda :
\pi_1(\mathcal{C}_{g,n/R}, \bar{x})
\;\xrightarrow{\;\sim\;}
\pi_1(\mathcal{C}_{g,n/R}, \bar{y}).
\]
This induces isomorphisms
\(\lambda' :
\pi_1'(\mathcal{C}_{g,n/R}, \bar{x}) \xrightarrow{\sim}
\pi_1'(\mathcal{C}_{g,n/R}, \bar{y})\),
\(\bar\lambda :
\pi_1(\mathcal{M}_{g,n/R}, \bar{\eta}) \xrightarrow{\sim}
\pi_1(\mathcal{M}_{g,n/R}, \bar{\xi})\), and 
\(\tilde{\lambda} :
\pi_1(C_{\bar{\eta}}, \bar{x})^{\prol} \xrightarrow{\sim}
\pi_1(C_{\bar{\xi}}, \bar{y})^{\prol}\),
which together make the following diagram commute:
\begin{equation}\label{eq:specialization-diagram}
\xymatrix@C=10pt@R=10pt{
1 \ar[r] &
\pi_1(C_{\bar{\xi}}, \bar{y})^{\prol} \ar[r] \ar@{=}[d] &
\pi_1'(\mathcal{C}_{g,n/\bar{k}}, \bar{y}) \ar[r] \ar[d] &
\pi_1(\mathcal{M}_{g,n/\bar{k}}, \bar{\xi}) \ar[r] \ar[d] &
1 \\
1 \ar[r] &
\pi_1(C_{\bar{\xi}}, \bar{y})^{\prol} \ar[r]  &
\pi_1'(\mathcal{C}_{g,n/R}, \bar{y}) \ar[r]  &
\pi_1(\mathcal{M}_{g,n/R}, \bar{\xi}) \ar[r]  &
1 \\
1 \ar[r] &
\pi_1(C_{\bar{\eta}}, \bar{x})^{\prol} \ar[r] \ar@{=}[d] \ar[u]^{\tilde{\lambda}}_\cong&
\pi_1'(\mathcal{C}_{g,n/R}, \bar{x}) \ar[r]\ar[u]^{\lambda'}_\cong &
\pi_1(\mathcal{M}_{g,n/R}, \bar{\eta}) \ar[r]\ar[u]^{\bar{\lambda}}_\cong &
1 \\
1 \ar[r] &
\pi_1(C_{\bar{\eta}}, \bar{x})^{\prol} \ar[r] &
\pi_1'(\mathcal{C}_{g,n/\bar{K}}, \bar{x}) \ar[r] \ar[u] &
\pi_1(\mathcal{M}_{g,n/\bar{K}}, \bar{\eta}) \ar[r] \ar[u] &
1 ,
}
\end{equation}
where rows are exact (see \S \ref{hom seq in p}).
Observe that the monodromy representation
\(\rho_{\mathcal C_{g,n}} :
\pi_1(\mathcal C_{g,n/\bar k}) \to \Sp(H_{\mathbb Z_\ell})\)
factors through \(\pi_1(\mathcal M_{g,n/\bar k})\), and hence through the quotient
\(\pi_1'(\mathcal C_{g,n/\bar k})\). It follows that there is a unique homomorphism
\(\pi_1(\mathcal C_{g,n/\bar k})^{\rel} \to
\pi_1'(\mathcal C_{g,n/\bar k})^{\rel}\) by naturality of relative pro-$\ell$ completion. On the other hand, we have a commutative diagram
\[
\xymatrix@R=9pt{
1 \ar[r] & N \ar[r] \ar[d] &
\pi_1(\mathcal C_{g,n/\bar k}) \ar[r] \ar[d] &
\pi_1'(\mathcal C_{g,n/\bar k}) \ar[r] \ar[d] &
1 \\
1 \ar[r] & K \ar[r] &
\pi_1(\mathcal C_{g,n/\bar k})^{\rel} \ar[r] &
\Sp(H_{\mathbb Z_\ell}) \ar[r] &
1 ,
}
\]
where \(N\) is defined in \S\ref{hom seq in p}. Since \(K\) is pro-\(\ell\), the
image of \(N\) in \(\pi_1(\mathcal C_{g,n/\bar k})^{\rel}\) is trivial. Hence the
completion \(\pi_1(\mathcal C_{g,n/\bar k}) \to
\pi_1(\mathcal C_{g,n/\bar k})^{\rel}\) factors through
\(\pi_1'(\mathcal C_{g,n/\bar k})\), inducing a unique map
\(\pi_1'(\mathcal C_{g,n/\bar k})^{\rel} \to
\pi_1(\mathcal C_{g,n/\bar k})^{\rel}\).

By the universal property of relative pro-\(\ell\) completion, these two maps are
inverse isomorphisms, and thus
\(\pi_1(\mathcal C_{g,n/\bar k})^{\rel} \cong
\pi_1'(\mathcal C_{g,n/\bar k})^{\rel}\).
The same argument applies over \(R\) and \(\bar K\).

Applying relative pro-\(\ell\) completion to the diagram
\eqref{eq:specialization-diagram}, we obtain
\begin{equation}\label{eq:specialization-diagram-rel}
\xymatrix@C=10pt@R=10pt{
1 \ar[r] &
\pi_1(C_{\bar{\xi}}, \bar{y})^{\prol} \ar[r] \ar@{=}[d] &
\pi_1(\mathcal{C}_{g,n/\bar{k}}, \bar{y})^{\rel} \ar[r] \ar[d]^{\cong} &
\pi_1(\mathcal{M}_{g,n/\bar{k}}, \bar{\xi})^{\rel} \ar[r] \ar[d]^{\cong} &
1 \\
1 \ar[r] &
\pi_1(C_{\bar{\xi}}, \bar{y})^{\prol} \ar[r]  &
\pi_1(\mathcal{C}_{g,n/R}, \bar{y})^{\rel} \ar[r]  &
\pi_1(\mathcal{M}_{g,n/R}, \bar{\xi})^{\rel} \ar[r]  &
1 \\
1 \ar[r] &
\pi_1(C_{\bar{\eta}}, \bar{x})^{\prol} \ar[r] \ar@{=}[d]\ar[u]^{\tilde{\lambda}}_\cong &
\pi_1(\mathcal{C}_{g,n/R}, \bar{x})^{\rel} \ar[r]\ar[u]^{\lambda^{'\rel}}_\cong  &
\pi_1(\mathcal{M}_{g,n/R}, \bar{\eta})^{\rel} \ar[r]\ar[u]^{\bar\lambda^\rel}_\cong &
1 \\
1 \ar[r] &
\pi_1(C_{\bar{\eta}}, \bar{x})^{\prol} \ar[r] &
\pi_1(\mathcal{C}_{g,n/\bar{K}}, \bar{x})^{\rel} \ar[r] \ar[u]_{\cong} &
\pi_1(\mathcal{M}_{g,n/\bar{K}}, \bar{\eta})^{\rel} \ar[r] \ar[u]_{\cong} &
1 .
}
\end{equation}
The rows are exact by right exactness of relative pro-\(\ell\) completion and the
center-freeness of \(\pi_1(C_{\bar\xi})^{\prol} (\cong \pi_1(C_{\bar\eta})^{\prol}) \), and the right vertical homomorphisms are
isomorphisms (see the proof of \cite[Thm.~5.6]{wat_rational_points_inp}).
It therefore follows that all middle vertical homomorphisms are isomorphisms.
\end{proof}

\subsection{Specialization for configuration spaces in the ordered case}

Let \(L\) be an algebraically closed field of characteristic \(0\).
Let \(\bar\eta\) be a geometric point of \(\mathcal M_{g/\bar L}\), and let
\(C_{\bar\eta}\) be the fiber of
\(\pi_{g/L} : \mathcal C_{g/L} \to \mathcal M_{g/L}\) over \(\bar\eta\).
Fix a geometric point \(\bar x \in F(C_{\bar\eta},n)\).
Recall that \(\pi_1(F(C_{\bar\eta},n),\bar x)\) is naturally isomorphic to the
profinite completion \(\widehat P_{g,n}\); \textcolor{black}{we identify}
\(\pi_1(F(C_{\bar\eta},n),\bar x)\) \textcolor{black}{with \(\widehat P_{g,n}\)}.

\textcolor{black}{The forgetful morphism \(\mathcal M_{g,n/L}\to \mathcal M_{g/L}\) yields an exact sequence}
\[
1 \longrightarrow \widehat P_{g,n}
\longrightarrow \pi_1(\mathcal M_{g,n/ L},\bar x)
\longrightarrow \pi_1(\mathcal M_{g/ L},\bar\eta)
\longrightarrow 1 .
\]
Injectivity on the left follows from the center-freeness of \(\widehat P_{g,n}\),
which can be proved using \cite{Anderson1974} and \cite[Lem.~4.2.2]{Birman1974}.
Conjugation in $\pi_1(\mathcal M_{g,n/L})$ gives a continuous action on
$\widehat P_{g,n}$, which induces a continuous action on $P_{g,n}^{\prol}$.
Consequently, there is a natural homomorphism
\[
\rho^{\prol} :
\pi_1(\mathcal M_{g,n/ L})
\longrightarrow
\Aut_{\cts}(P_{g,n}^{\prol}).
\]

\begin{lemma}\label{lem:rho-prol-factors-rel}
The homomorphism \(\rho^\prol:\pi_1(\M_{g,n/L})\to \Aut_\cts(P_{g,n}^\prol)\)
factors through the relative pro-\(\ell\) completion
\(\pi_1(\M_{g,n/L})^{\rel}\).
\end{lemma}

\begin{proof}
Denote the image of \(\rho^\prol\) by \(G\subset \Aut_\cts(P_{g,n}^\prol)\).
The natural monodromy representation
\(\pi_1(\M_{g,n/L})\to \Sp(H_{\Zl})\)
\textcolor{black}{factors through \(G\), yielding a continuous homomorphism}
\(\psi: G \to \Sp(H_{\Zl})\).

Since \(P_{g,n}\) is finitely generated and \(P_{g,n}^{\ab}\) has no \(\ell\)-primary
torsion, the canonical map \(P_{g,n}\to P_{g,n}^{\prol}\) induces
\[
H_1(P_{g,n}^{\prol}, \Zl)\cong H_1(P_{g,n},\Z)\otimes \Z_\ell.
\]
\textcolor{black}{Let}
\[
\Aut^0_\cts(P_{g,n}^\prol)
:=
\ker\!\Bigl(\Aut_\cts(P_{g,n}^\prol)\to
\Aut_\cts\bigl(H_1(P_{g,n})\otimes\Z_\ell\bigr)\Bigr).
\]
\textcolor{black}{A standard fact implies that \(\Aut^0_\cts(P_{g,n}^\prol)\) is a pro-\(\ell\) group.}
Moreover, the image of \(G\) in
\(\Aut_\cts(H_1(P_{g,n})\otimes\Zl)\)
\textcolor{black}{is the diagonal copy of \(\Sp(H_{\Zl})\)}. Thus we have a commutative
diagram with exact rows
\[
\xymatrix@R=11pt@C=9pt{
1 \ar[r]
  & \ker\psi \ar@{^{(}->}[d] \ar[r]
  & G \ar@{^{(}->}[d] \ar[r]^{\psi}
  & \Sp(H_{\Zl}) \ar@{^{(}->}[d]^{\diag} \ar[r]
  & 1 \\
1 \ar[r]
  & \Aut^0_\cts(P_{g,n}^\prol) \ar[r]
  & \Aut_\cts(P_{g,n}^\prol) \ar[r]
  & \Aut_\cts(H_1(P_{g,n})\otimes\Zl) \ar[r]
  & 1 .
}
\]
\textcolor{black}{In particular, \(\ker\psi\subset \Aut^0_\cts(P_{g,n}^\prol)\) is a pro-\(\ell\) group.}
By the universal property of relative pro-\(\ell\) completion, this implies that
\(\rho^\prol\) factors through \(\pi_1(\M_{g,n/L})^{\rel}\), i.e.\ there exists a
homomorphism \(\pi_1(\M_{g,n/L})^{\rel}\to G\) whose composition with the inclusion
\(G\hookrightarrow \Aut_\cts(P_{g,n}^\prol)\) equals \(\rho^\prol\).
\end{proof}

Recall the notation from \S\ref{pro-l comp of F(C,n)}. \textcolor{black}{Now return to the mixed-characteristic model \(X/R\) used to define specialization.}
Let \(C\) be a smooth proper curve over \(\bar k\), and let
\(X \to \Spec R\) be a smooth proper \(R\)-scheme whose special fiber is
isomorphic to \(C\). There is a unique morphism
\(\psi : \Spec R \to \mathcal M_{g/R}\) such that the pullback of the universal
curve \(\mathcal C_g \to \mathcal M_g\) along \(\psi\) is isomorphic to
\(X \to \Spec R\); we denote this pullback by \(X_R\). Let
\(\phi : X_R \to \mathcal C_{g/R}\) be the induced morphism.

The morphism \(\phi : X_R \to \mathcal C_{g/R}\) is a monomorphism; hence the
induced morphism \(\phi^n : X_R^n \to \mathcal C_{g/R}^n\) preserves the complement
of the fat diagonal and therefore induces a morphism
\[
F(\phi) : F(X_R,n) \longrightarrow F(\mathcal C_{g/R},n).
\]
Since \(\mathcal M_{g,n/R}\) is canonically isomorphic to
\(F(\mathcal C_{g/R},n)\), we obtain a sequence of morphisms
\begin{equation}\label{eq:FXR-to-Mgn-to-Mg}
F(X_R,n) \xrightarrow{\,F(\phi)\,} \mathcal M_{g,n/R}
\longrightarrow \mathcal M_{g/R}.
\end{equation}

Let \(\bar\xi\) and \(\bar\eta\) be geometric points of
\(\mathcal M_{g/\bar k}\) and \(\mathcal M_{g/\bar K}\), respectively. Via the structure morphism \(\mathcal M_{g/R}\to \Spec R\), we also regard
\(\bar\xi\) and \(\bar\eta\) as the closed geometric point and the generic geometric
point of \(\Spec R\), respectively.

Over \(\bar\eta\), we have
\(F(X_R,n)_{\bar\eta}\cong F(X_{\bar\eta},n)\), and there is an exact sequence of
étale fundamental groups
\[
1 \longrightarrow
\pi_1\bigl(F(X_{\bar\eta},n),\bar x\bigr)
\longrightarrow
\pi_1\bigl(\mathcal M_{g,n/\bar K},\bar x\bigr)
\longrightarrow
\pi_1\bigl(\mathcal M_{g/\bar K},\bar\eta\bigr)
\longrightarrow 1 .
\]

On the other hand, over the geometric point \(\bar\xi\), we have
\(F(X_R,n)_{\bar\xi}\cong F(C,n)\), and there is a sequence
\[
\pi_1\bigl(F(C,n),\bar y\bigr)
\longrightarrow
\pi_1\bigl(\mathcal M_{g,n/\bar k},\bar y\bigr)
\longrightarrow
\pi_1\bigl(\mathcal M_{g/\bar k},\bar\xi\bigr)
\longrightarrow 1,
\]
where the composition is trivial and the second homomorphism is surjective.

\begin{proposition}\label{prop:ordered-specialization-rel-prol}
For $g\ge 2$ and $n\ge 1$, there is a commutative diagram:
\[
\xymatrix@R=9pt@C=9pt{
1 \ar[r] &
\pi_1(F(C,n), \bar{y})^{\prol} \ar[r]  &
\pi_1(\mathcal{M}_{g,n/\bar k}, \bar{y})^{\rel} \ar[r]  &
\pi_1(\mathcal{M}_{g/\bar k}, \bar{\xi})^{\rel} \ar[r]  &
1 \\
1 \ar[r] &
P^{\prol}_{g,n} \ar[r] \ar[u]_\cong&
\pi_1(\mathcal{M}_{g,n/\bar{K}}, \bar{x})^{\rel} \ar[r] \ar[u]_{\cong} &
\pi_1(\mathcal{M}_{g/\bar{K}}, \bar{\eta})^{\rel} \ar[r] \ar[u]_{\cong} &
1 ,
}
\]
where the rows are exact and all vertical maps are isomorphisms.
\end{proposition}

\begin{proof}
Fix the base-point change isomorphism
\[
\lambda:\pi_1\bigl(F(X_R,n),\bar x\bigr)\xrightarrow{\cong}
\pi_1\bigl(F(X_R,n),\bar y\bigr),
\]
where \(\bar x\) (resp.\ \(\bar y\)) lies over \(\bar\eta\) (resp.\ \(\bar\xi\)).
Via the sequence \eqref{eq:FXR-to-Mgn-to-Mg}, the isomorphism \(\lambda\) induces
isomorphisms
\[
\lambda':
\pi_1\bigl(\mathcal M_{g,n/R},\bar x\bigr)\xrightarrow{\cong}
\pi_1\bigl(\mathcal M_{g,n/R},\bar y\bigr),
\qquad
\lambda'':
\pi_1\bigl(\mathcal M_{g/R},\bar\eta\bigr)\xrightarrow{\cong}
\pi_1\bigl(\mathcal M_{g/R},\bar\xi\bigr).
\]
These fit into a commutative diagram
\[
\xymatrix@R=9pt{
\pi_1(F(X_R,n),\bar x)\ar[r]^{F(\phi)_*}\ar[d]_{\lambda}^{\cong}&
\pi_1(\mathcal M_{g,n/R},\bar x)\ar[r]\ar[d]_{\lambda'}^{\cong}&
\pi_1(\mathcal M_{g/R},\bar\eta)\ar[d]_{\lambda''}^{\cong}\\
\pi_1(F(X_R,n),\bar y)\ar[r]^{F(\phi)_*}&
\pi_1(\mathcal M_{g,n/R},\bar y)\ar[r]&
\pi_1(\mathcal M_{g/R},\bar\xi).
}
\]

Passing to fibers over the generic and special geometric points yields the
following commutative diagram, whose top row is exact:
\[
\xymatrix@R=9pt{
1\ar[r]&
\pi_1(F(X_{\bar\eta},n),\bar x)\ar[r]\ar[d]&
\pi_1(\mathcal M_{g,n/\bar K},\bar x)\ar[r]\ar[d]&
\pi_1(\mathcal M_{g/\bar K},\bar\eta)\ar[r]\ar[d]&
1\\
&
\pi_1(F(X_R,n),\bar x)\ar[r]\ar[d]_{\lambda}^{\cong}&
\pi_1(\mathcal M_{g,n/R},\bar x)\ar[r]\ar[d]_{\lambda'}^{\cong}&
\pi_1(\mathcal M_{g/R},\bar\eta)\ar[r]\ar[d]_{\lambda''}^{\cong}&
1\\
&
\pi_1(F(X_R,n),\bar y)\ar[r]&
\pi_1(\mathcal M_{g,n/R},\bar y)\ar[r]&
\pi_1(\mathcal M_{g/R},\bar\xi)\ar[r]&
1\\
&
\pi_1(F(C,n),\bar y)\ar[r]\ar[u]&
\pi_1(\mathcal M_{g,n/\bar k},\bar y)\ar[r]\ar[u]&
\pi_1(\mathcal M_{g/\bar k},\bar\xi)\ar[r]\ar[u]&
1.
}
\]

Applying relative pro-\(\ell\) completion, we obtain a commutative diagram
\[
\xymatrix@R=9pt@C=9pt{
1\ar[r]&
\pi_1(F(X_{\bar\eta},n),\bar x)^{\prol}\ar[r]\ar[d]^{\cong}&
\pi_1(\mathcal M_{g,n/\bar K},\bar x)^{\rel}\ar[r]\ar[d]^{\cong}&
\pi_1(\mathcal M_{g/\bar K},\bar\eta)^{\rel}\ar[r]\ar[d]^{\cong}&
1\\
1\ar[r]&
\pi_1(F(X_R,n),\bar x)^{\prol}\ar[r]\ar[d]_{\lambda}^{\cong}&
\pi_1(\mathcal M_{g,n/R},\bar x)^{\rel}\ar[r]\ar[d]_{\lambda'}^{\cong}&
\pi_1(\mathcal M_{g/R},\bar\eta)^{\rel}\ar[r]\ar[d]_{\lambda''}^{\cong}&
1\\
1\ar[r]&
\pi_1(F(X_R,n),\bar y)^{\prol}\ar[r]&
\pi_1(\mathcal M_{g,n/R},\bar y)^{\rel}\ar[r]&
\pi_1(\mathcal M_{g/R},\bar\xi)^{\rel}\ar[r]&
1\\
1\ar[r]&
\pi_1(F(C,n),\bar y)^{\prol}\ar[r]\ar[u]_{\cong}&
\pi_1(\mathcal M_{g,n/\bar k},\bar y)^{\rel}\ar[r]\ar[u]_{\cong}&
\pi_1(\mathcal M_{g/\bar k},\bar\xi)^{\rel}\ar[r]\ar[u]_{\cong}&
1.
}
\]

The left vertical maps are the specialization isomorphisms defined in
\S\ref{pro-l comp of F(C,n)}. The middle and right vertical maps are isomorphisms
by the proof of \cite[Thm.~5.6]{wat_rational_points_inp}.

It remains to justify exactness of the top row. By Lemma~\ref{lem:rho-prol-factors-rel},
the conjugation action of \(\pi_1(\mathcal M_{g,n/\bar K})\) on
\(\pi_1(F(X_{\bar\eta},n))\) induces an action of
\(\pi_1(\mathcal M_{g,n/\bar K})^{\rel}\) on
\(\pi_1(F(X_{\bar\eta},n))^{\prol}\).
Let \(W_\bullet\) denote the weight filtration on
\(\pi_1(F(X_{\bar\eta},n))^{\prol}\) constructed in \cite{NTU}.
Since the associated graded Lie algebra
\(\Gr^W_\bullet \pi_1(F(X_{\bar\eta},n))^{\prol}\) has trivial center and the
filtration \(W_\bullet\) is separated
(\cite[Cor.~2.6, p.~201]{NTU}), it follows that
\(\pi_1(F(X_{\bar\eta},n))^{\prol}\) is center-free. Consequently, the induced homomorphism
\[
\pi_1(F(X_{\bar\eta},n))^{\prol}
\longrightarrow
\pi_1(\mathcal M_{g,n/\bar K})^{\rel}
\]
is injective. This proves exactness of the top row, and hence all rows in the
diagram are exact. The proposition follows.
\end{proof}

\subsection{Homotopy exact sequences in the unordered case}

We first compare the ordered and unordered universal curves. Consider the
commutative diagram
\[
\xymatrix@C=10pt@R=10pt{
\cC_{g,n/\bar k}\ar[r]\ar[d]&\M_{g,n/\bar k}\ar[d]\\
\cC_{g,[n]/\bar k}\ar[r]&\M_{g,[n]/\bar k},
}
\]
which induces a commutative diagram of fundamental groups
\begin{equation}\label{eq:homotopy-ordered-unordered-curve}
\xymatrix@C=10pt@R=10pt{
\pi_1\!\bigl(C,\bar y\bigr) \ar[r]\ar@{=}[d]
  & \pi_1(\cC_{g,n/\bar k},\bar y) \ar[r]\ar[d]
  & \pi_1(\M_{g,n/\bar k},\bar\xi) \ar[d] \ar[r]
  & 1 \\
\pi_1\!\bigl(C,\bar y\bigr) \ar[r]
  & \pi_1(\cC_{g,[n]/\bar k},\bar y) \ar[r]
  & \pi_1(\M_{g,[n]/\bar k},\bar\xi) \ar[r]
  & 1 ,
}
\end{equation}
where the rows are exact on the right.

Since the $S_n$-action permutes the marked points but acts trivially on the
underlying curve, the local system $\mathbb H_{\mathbb Z_\ell}$ descends to
$\M_{g,[n]}$. Consequently, the representation
\[
\rho_{\mathbb Z_\ell,\bar\xi} :
\pi_1(\M_{g,n/\bar k},\bar\xi) \longrightarrow \Sp(H_{\mathbb Z_\ell})
\]
factors through $\rho_{\Zl,[\bar\xi]}:\pi_1(\M_{g,[n]/\bar k},\bar\xi)\to\Sp(H_\Zl)$.
Thus, there is a commutative diagram
\begin{equation}\label{eq:monodromy-ordered-unordered-curve}
\xymatrix@C=13pt@R=10pt{
 \pi_1(\cC_{g,n/\bar k},\bar y) \ar[r] \ar[d]
  & \pi_1(\M_{g,n/\bar k},\bar\xi) \ar[r]^-{\rho_{\Zl,\bar\xi}} \ar[d]
  & \Sp(H_{\mathbb Z_\ell}) \ar@{=}[d] \\
 \pi_1(\cC_{g,[n]/\bar k},\bar y) \ar[r]
  & \pi_1(\M_{g,[n]/\bar k},\bar\xi) \ar[r]^-{\rho_{\Zl,[\bar\xi]}}
  & \Sp(H_{\mathbb Z_\ell}).
}
\end{equation}
Since $\pi_1(C)^{\prol}$ is center-free, applying relative pro-$\ell$
completion to \eqref{eq:monodromy-ordered-unordered-curve} with respect to the
representations into $\Sp(H_{\mathbb Z_\ell})$ yields the following result.

\begin{proposition}\label{prop:ordered-unordered-homotopy-rel-prol}
For $g\ge 2$ and $n\ge 0$, the diagram~\eqref{eq:homotopy-ordered-unordered-curve} induces a commutative
diagram
\begin{equation*}
\xymatrix@C=10pt@R=10pt{
1 \ar[r] &
\pi_1(C)^{\prol} \ar[r] \ar@{=}[d] &
\pi_1(\cC_{g,n/\bar k})^{\rel} \ar[r] \ar[d] &
\pi_1(\M_{g,n/\bar k})^{\rel} \ar[r] \ar[d] &
1 \\
1 \ar[r] &
\pi_1(C)^{\prol} \ar[r] &
\pi_1(\cC_{g,[n]/\bar k})^{\rel} \ar[r] &
\pi_1(\M_{g,[n]/\bar k})^{\rel} \ar[r] &
1 ,
}
\end{equation*}
whose rows are exact.
\end{proposition}

\textcolor{black}{We next compare the ordered and unordered configuration spaces.}
Consider the commutative diagram
\[
\xymatrix@C=10pt@R=10pt{
\mathcal M_{g,n/\bar k} \ar[r]\ar[d]
  & \mathcal M_{g/\bar k} \ar@{=}[d] \\
\mathcal M_{g,[n]/\bar k} \ar[r]
  & \mathcal M_{g/\bar k},
}
\]
which induces a commutative diagram of fundamental groups
\begin{equation}\label{eq:pi1-ordered-unordered}
\xymatrix@C=10pt@R=10pt{
\pi_1\!\bigl(F(C,n),\bar y\bigr) \ar[r]\ar[d]
  & \pi_1(\mathcal M_{g,n/\bar k},\bar y) \ar[r]\ar[d]
  & \pi_1(\mathcal M_{g/\bar k},\bar\xi) \ar@{=}[d] \ar[r]
  & 1 \\
\pi_1\!\bigl(F(C,[n]),\bar y\bigr) \ar[r]
  & \pi_1(\mathcal M_{g,[n]/\bar k},\bar y) \ar[r]
  & \pi_1(\mathcal M_{g/\bar k},\bar\xi) \ar[r]
  & 1 ,
}
\end{equation}
whose rows are exact on the right. In what follows, we suppress basepoint
notation for fundamental groups whenever it is clear from the context.

\begin{proposition}\label{prop:ordered-unordered-rel-prol}
Suppose $g\ge 2$ and $n\ge 1$. Let $\ell\neq p$ be an odd prime.
There is a commutative diagram
\[
\xymatrix@R=9pt@C=9pt{
1 \ar[r] &
\pi_1\!\bigl(F(C,n)\bigr)^{\prol} \ar[r] \ar[d] &
\pi_1(\mathcal M_{g,n/\bar k})^{\rel} \ar[r]\ar[d] &
\pi_1(\mathcal M_{g/\bar k})^{\rel} \ar[r] \ar@{=}[d] &
1 \\
&
\pi_1\!\bigl(F(C,[n])\bigr)^{\prol} \ar[r] &
\pi_1(\mathcal M_{g,[n]/\bar k})^{\rel} \ar[r] &
\pi_1(\mathcal M_{g/\bar k})^{\rel} \ar[r] &
1 ,
}
\]
whose rows are exact.
\end{proposition}

\begin{proof}
Consider the exact sequence
\[
1 \longrightarrow \pi_1(\M_{g,n/\bar k})
\longrightarrow \pi_1(\M_{g,[n]/\bar k})
\longrightarrow S_n \longrightarrow 1,
\]
which fits into the commutative diagram
\begin{equation}\label{diag:ordered-unordered-monodromy}
\xymatrix@C=10pt@R=12pt{
1 \ar[r] &
\pi_1(\M_{g,n/\bar k}) \ar[r] \ar[d]^{\rho_{\mathbb Z_\ell, \bar\xi}} &
\pi_1(\M_{g,[n]/\bar k}) \ar[r] \ar[d]^{\rho_{\mathbb Z_\ell, [\bar\xi]}} &
S_n \ar[r] \ar[d] &
1 \\
1 \ar[r] &
\Sp(H_{\mathbb Z_\ell}) \ar@{=} [r] &
\Sp(H_{\mathbb Z_\ell}) \ar[r] &
1 &
}
\end{equation}
For \(n \ge 2\) and \(\ell\) an odd prime, the pro-\(\ell\) completion of \(S_n\)
is trivial. Since relative pro-\(\ell\) completion is right exact
\cite[Prop.~2.4]{HainMatsumoto2009}, applying it to
\eqref{diag:ordered-unordered-monodromy} yields an exact sequence
\[
\pi_1(\M_{g,n/\bar k})^{\rel}
\longrightarrow
\pi_1(\M_{g,[n]/\bar k})^{\rel}
\longrightarrow
S_n^{\prol} \cong 1.
\]
Hence the induced map
\(
\pi_1(\M_{g,n/\bar k})^{\rel}
\to
\pi_1(\M_{g,[n]/\bar k})^{\rel}
\)
is surjective.

Applying relative pro-\(\ell\) completion to \eqref{eq:pi1-ordered-unordered}, we
obtain a commutative diagram
\[
\xymatrix@R=9pt@C=9pt{
1 \ar[r] &
\pi_1\!\bigl(F(C,n)\bigr)^{\prol} \ar[r] \ar[d] &
\pi_1(\mathcal M_{g,n/\bar k})^{\rel} \ar[r] \ar[d] &
\pi_1(\mathcal M_{g/\bar k})^{\rel} \ar[r] \ar@{=}[d] &
1 \\
&
\pi_1\!\bigl(F(C,[n])\bigr)^{\prol} \ar[r] &
\pi_1(\mathcal M_{g,[n]/\bar k})^{\rel} \ar[r] &
\pi_1(\mathcal M_{g/\bar k})^{\rel} \ar[r] &
1 .
}
\]
The first row is exact by Proposition~\ref{prop:ordered-specialization-rel-prol}.
By right exactness of the completion, the second row is exact on the right.
Since the middle vertical map is surjective, the second row is exact in the
middle as well.
\end{proof}

\section{Continuous Relative Completion}\label{cont rel comp}
The continuous relative completion is the key tool used in the prequel to this
paper~\cite{LW25}. It is the profinite analogue of the relative completion of a
discrete group. Detailed introductions to the definition and properties of
relative completion can be found in
\cite{hain_completion, hain_hodge_rel, hain_relwtfilt}.

Let $\G$ be a profinite group, let $R$ be a reductive group over $\Q_\ell$, and
let $\rho:\G\to R(\Q_\ell)$ be a continuous homomorphism with Zariski-dense
image. The continuous relative completion of $\G$ with respect to $\rho$ is a
pro-algebraic $\Q_\ell$-group $\cG$, together with a continuous Zariski-dense
homomorphism $\tilde\rho:\G\to\cG(\Q_\ell)$ lifting $\rho$, such that $\cG$ fits
into an exact sequence
\[
1 \longrightarrow \U \longrightarrow \cG \longrightarrow R \longrightarrow 1,
\]
where $\U$ is a prounipotent $\Q_\ell$-group.

The group $\cG$ is characterized by the following universal property: if $G$ is
a pro-algebraic $\Q_\ell$-group that is an extension of $R$ by a prounipotent
$\Q_\ell$-group and if $\psi:\G\to G(\Q_\ell)$ is a continuous homomorphism
lifting $\rho$, then there exists a unique morphism of $\Q_\ell$-groups
$\phi:\cG\to G$ such that the diagram
\[
\xymatrix{
\G \ar[r]^{\tilde\rho} \ar[d]_{\psi} &
\cG(\Q_\ell) \ar[d]\ar[ld]_{\phi(\Q_\ell)} \\
G(\Q_\ell) \ar[r] &
R(\Q_\ell)
}
\]
commutes.
\subsection{Continuous relative completion of
$\pi_1(\cC_{g,n/\bar\Q}) \to \pi_1(\M_{g,n/\bar\Q})$}

We briefly recall the structure of the continuous relative completions of
$\pi_1(\M_{g,n/\bar\Q})$ and $\pi_1(\cC_{g,n/\bar\Q})$ associated with the
universal curve, following the framework of Hain and Matsumoto.

Consider the $\ell$-adic local system
\[
\mathbb H_{\Q_\ell} := \mathbb H_{\mathbb Z_\ell} \otimes \Q_\ell
\]
on $\mathcal M_{g,n/\mathbb Z}$.
Let $\bar\zeta$ be a geometric point of $\M_{g,n/\bar\Q}$, and let $\bar a$ be a
geometric point of $\cC_{g,n/\bar\Q}$ lying over~$\bar\zeta$.
Write $H_{\Q_\ell}$ for the fiber of $\mathbb H_{\Q_\ell}$ over~$\bar\zeta$.

The associated monodromy representation
\[
\rho_{\bar\zeta} :
\pi_1(\M_{g,n/\bar\Q},\bar \zeta) \longrightarrow \Sp(H_{\Q_\ell})
\]
is continuous with Zariski-dense image.
Composing $\rho_{\bar\zeta}$ with the homomorphism induced by the universal curve,
$\pi_1(\cC_{g,n/\bar\Q},\bar a)
\longrightarrow
\pi_1(\M_{g,n/\bar\Q},\bar \zeta)$,
we obtain a representation
\[
\rho_{\cC,\bar\zeta} :
\pi_1(\cC_{g,n/\bar\Q},\bar a)
\longrightarrow
\Sp(H_{\Q_\ell}),
\]
which corresponds to the pullback of the local system $\mathbb H_{\Q_\ell}$
along the universal curve.

Denote by $\cG_{g,n}$ and $\cG_{\cC_{g,n}}$ the continuous relative completions of
$\pi_1(\M_{g,n/\bar\Q})$ and $\pi_1(\cC_{g,n/\bar\Q})$ with respect to
$\rho_{\bar\zeta}$ and $\rho_{\cC,\bar\zeta}$, respectively.
Let $\U_{g,n}$ and $\U_{\cC_{g,n}}$ be their unipotent radicals.
We write
\[
\g_{g,n},\ \g_{\cC_{g,n}},\ \u_{g,n},\ \u_{\cC_{g,n}}
\]
for the Lie algebras of
$\cG_{g,n}$, $\cG_{\cC_{g,n}}$, $\U_{g,n}$, and $\U_{\cC_{g,n}}$, respectively.

\textcolor{black}{These Lie algebras carry natural weight filtrations, inherited from
the corresponding relative completions over~$\Q$ constructed using Hodge-theoretic
methods.}
\textcolor{black}{We use only the induced structural properties of these
filtrations}
(cf.~\cite[Thm.~13.1]{hain_hodge_rel}).

\subsection{The relative completion and relative pro-$\ell$ completion}
Suppose that $\Gamma$ is a profinite group and let
\[
\rho_{\mathbb Z_\ell} : \Gamma \longrightarrow R(\mathbb Z_\ell)
\]
be a continuous homomorphism whose extension
\[
\rho_{\mathbb Q_\ell} : \Gamma \longrightarrow R(\mathbb Q_\ell)
\]
has Zariski--dense image. Let \textcolor{black}{$\G^\rel$} denote the relative
pro-$\ell$ completion of $\Gamma$ with respect to $\rho_\Zl$, and write
\[
\rho^{\rel}_{\mathbb Z_\ell} : \Gamma^{\rel}
\longrightarrow R(\mathbb Z_\ell)
\]
for the induced homomorphism. Since
$\rho(\Gamma) \subseteq \rho^{\rel}_\Zl(\Gamma^{\rel})$,
the induced map
\[
\rho^\rel_{\Ql}:\Gamma^{\rel} \longrightarrow R(\mathbb Q_\ell)
\]
also has Zariski--dense image.
The following observation is then an immediate
consequence of the universal property.
\begin{proposition}\label{prop:compatibility-relative-completions}
The continuous relative completion of $\Gamma^{\rel}$ with respect to
the representation $\rho^{\rel}_\Ql$ is canonically isomorphic to the
continuous relative completion of $\Gamma$ with respect to $\rho_\Ql$.
\end{proposition}
\begin{proof}
Let $\cG$ and $\cG^{\rel}$ be the continuous relative completions of $\G$ and $\G^{\rel}$ with respect to $\rho_\Ql$ and $\rho^{\rel}_\Ql$, respectively.
The surjection $\G\to \G^{\rel}$ induces a unique surjective morphism 
$\phi:\cG\to \cG^{\rel}$.
Let $\bar \G$ be the image of $\G$ in $\cG$. Then $\bar \G$ is an extension of the image of $\rho_{\Zl}$ by a pro-$\ell$ group:
\[
\xymatrix@C=10pt@R=10pt{
1 \ar[r] 
  & \ker\rho_{\Zl} \ar[r] \ar@^{(->}[d] 
  & \bar \G \ar[r] \ar@^{(->}[d] 
  & \im\rho_{\Zl} \ar[r] \ar@^{(->}[d] 
  & 1 \\
1 \ar[r] 
  & U(\Q_\ell) \ar[r] 
  & \cG(\Q_\ell) \ar[r] 
  & R(\Q_\ell) \ar[r] 
  & 1.
}
\]
Since every compact subgroup of $U(\Q_\ell)$ is pro-$\ell$, the kernel
$\ker(\rho_{\mathbb Z_\ell})$ is a pro-$\ell$ group. Hence the map
$\G \to \bar{\G}$ factors through $\G^{\rel}$. Since the diagram
$$
\xymatrix@C=10pt@R=10pt{
\G^{\rel}\ar[dr]^{\rho^{\rel}_\Ql}\ar[d]\\
\cG(\Ql)\ar[r]&R(\Ql)
}
$$ commutes, there exists a unique surjective morphism 
$\psi:\cG^{\rel}\to \cG$. By the universal property of relative completion, $\phi$ and $\psi$ are mutually inverse, and therefore give a canonical isomorphism.  
\end{proof}

\subsection{Continuous relative completion in the ordered case}
 Consider the $\ell$-adic local system
$\mathbb H_{\Q_\ell}:=\mathbb H_{\mathbb Z_\ell}\otimes\Q_\ell$
on $\mathcal M_{g,n/\mathbb Z}$. With the notation of
\S\ref{subsec:homotopy-rel-prol}, write $H_{\Q_\ell,\bar\xi}$
(resp.\ $H_{\Q_\ell,\bar\eta}$) for its fiber over $\bar\xi$
(resp.\ $\bar\eta$). This yields symplectic representations
\[
\rho_{\Q_\ell,\bar\xi} :
\pi_1(\mathcal M_{g,n/\bar k},\bar\xi)
\longrightarrow
\Sp(H_{\Q_\ell,\bar\xi}),
\qquad
\rho_{\Q_\ell,\bar\eta} :
\pi_1(\mathcal M_{g,n/\bar K},\bar\eta)
\longrightarrow
\Sp(H_{\Q_\ell,\bar\eta}).
\]
Passing to relative pro-$\ell$ completions, we obtain the induced
representations
\[
\rho^{\rel}_{\Q_\ell,\bar\xi} :
\pi_1(\mathcal M_{g,n/\bar k},\bar\xi)^{\rel}
\longrightarrow
\Sp(H_{\Q_\ell,\bar\xi}),
\quad
\rho^{\rel}_{\Q_\ell,\bar\eta} :
\pi_1(\mathcal M_{g,n/\bar K},\bar\eta)^{\rel}
\longrightarrow
\Sp(H_{\Q_\ell,\bar\eta}).
\]

Denote by $\cG_{g,n,\bar\xi}$ (resp.\ $\cG_{g,n,\bar\eta}$) the continuous relative
completion of $\pi_1(\M_{g,n/\bar k},\bar\xi)$
(resp.\ $\pi_1(\M_{g,n/\bar K},\bar\eta)$) with respect to
$\rho_{\Q_\ell,\bar\xi}$ (resp.\ $\rho_{\Q_\ell,\bar\eta}$)
Similarly, denote by $\cG_{\cC_{g,n},\bar\xi}$
(resp.\ $\cG_{\cC_{g,n},\bar\eta}$) the continuous relative completion of
$\pi_1(\cC_{g,n/\bar k},\bar\xi)$
(resp.\ $\pi_1(\cC_{g,n/\bar K},\bar\eta)$) with respect to the induced
representations into $\Sp(H_{\Q_\ell,\bar\xi})$
(resp.\ $\Sp(H_{\Q_\ell,\bar\eta})$).

\begin{proposition}\label{prop:basechange-rel-completions}
There are natural isomorphisms of pro-algebraic $\Q_\ell$-groups
\[
\cG_{g,n,\bar\eta} \;\cong\; \cG_{g,n,\bar\xi}\;\cong\; \cG_{g,n}
\]
and 
\[
\cG_{\cC_{g,n},\bar\eta}\cong \textcolor{black}{\cG_{\cC_{g,n},\bar\xi}}\cong \cG_{\cC_{g,n}}.
\]
\end{proposition}

\begin{proof}
By Proposition~\ref{prop:specialization-rel-prol} and
Lemma~\ref{lem:H1-duality}, the representations
$\rho^{\rel}_{\Q_\ell,\bar\xi}$ and $\rho^{\rel}_{\Q_\ell,\bar\eta}$, when
composed with the natural morphisms
\[
\pi_1(\cC_{g,n/\bar k})^{\rel}\to \pi_1(\M_{g,n/\bar k})^{\rel},
\qquad
\pi_1(\cC_{g,n/\bar K})^{\rel}\to \pi_1(\M_{g,n/\bar K})^{\rel},
\]
induce, on $H_1$, the actions arising from conjugation on
$\pi_1(C_{\bar\xi})^{\prol}$ and $\pi_1(C_{\bar\eta})^{\prol}$, respectively.
Consequently, the specialization isomorphism yields a commutative diagram
\[
\xymatrix@C=10pt@R=12pt{
\pi_1(\cC_{g,n/\bar K},\bar\eta)^{\rel} \ar[r]^{\cong}
  \ar[d]
&
\pi_1(\cC_{g,n/\bar k},\bar\xi)^{\rel} \ar[d]\\
\pi_1(\M_{g,n/\bar K},\bar\eta)^{\rel} \ar[r]^{\cong}
  \ar[d]_{\rho^{\rel}_{\Q_\ell,\bar\eta}}
&
\pi_1(\M_{g,n/\bar k},\bar\xi)^{\rel} \ar[d]^{\rho^{\rel}_{\Q_\ell,\bar\xi}}\\
\Sp(H_{\Q_\ell,\bar\eta}) \ar[r]^{\cong}
&
\Sp(H_{\Q_\ell,\bar\xi}),
}
\]
where the bottom isomorphism is induced by
$\pi_1(C_{\bar\eta})^{\prol}\cong\pi_1(C_{\bar\xi})^{\prol}$.
Taking the continuous relative completion of the above diagram and using
Proposition~\ref{prop:compatibility-relative-completions}, which identifies the
resulting completion with the continuous relative completion taken with respect
to the original representations $\rho_{\Q_\ell,\bar\eta}$ and
$\rho_{\Q_\ell,\bar\xi}$, we obtain natural isomorphisms
\[
\cG_{g,n,\bar\eta} \;\cong\; \cG_{g,n,\bar\xi}
\quad\text{and}\quad
\cG_{\cC_{g,n},\bar\eta} \;\cong\; \cG_{\cC_{g,n},\bar\xi}.
\]
After fixing an embedding $\overline{\Q}\hookrightarrow\bar K$, invariance of
the étale fundamental group under extension of algebraically closed fields of
characteristic~$0$ gives canonical isomorphisms
\[
\pi_1(\M_{g,n/\bar K}) \;\xrightarrow{\;\sim\;}\;
\pi_1(\M_{g,n/\overline{\Q}})
\quad\text{and}\quad
\pi_1(\cC_{g,n/\bar K}) \;\xrightarrow{\;\sim\;}\;
\pi_1(\cC_{g,n/\overline{\Q}}).
\]
Applying continuous relative completion, we obtain canonical isomorphisms
\[
\cG_{g,n,\bar\eta} \;\cong\; \cG_{g,n}
\quad\text{and}\quad
 \cG_{\cC_{g,n},\bar\eta}\;\cong\; \cG_{\cC_{g,n}}.
\]

\end{proof}
\subsection{\textcolor{black}{Continuous relative completion in the unordered case}}

\textcolor{black}{We retain the notation from the ordered case and adapt it to the unordered setting.}
For simplicity, set $H_{\Q_\ell}:=H_{\Q_\ell,\bar\xi}$. The $\ell$-adic representation
$\rho_{\Q_\ell,\bar\xi} :
\pi_1(\M_{g,n/\bar k},\bar\xi)\longrightarrow \Sp(H_{\Q_\ell})$
factors through the natural homomorphism
$\pi_1(\M_{g,n/\bar k}) \longrightarrow \pi_1(\M_{g,[n]/\bar k})$,
and hence induces a representation
\[
\rho_{\Q_\ell,[\bar\xi]} :
\pi_1(\M_{g,[n]/\bar k},\bar\xi)\longrightarrow \Sp(H_{\Q_\ell}).
\]
Let
\[
\rho_{\cC_{\Q_\ell},[\bar\xi]} :
\pi_1(\cC_{g,[n]/\bar k},\bar x)
\longrightarrow
\pi_1(\M_{g,[n]/\bar k},\bar\xi)
\longrightarrow
\Sp(H_{\Q_\ell})
\]
\textcolor{black}{denote the induced representation obtained by pullback along the universal curve.}

Denote by $\cG_{g,[n],\bar\xi}$ and $\cG_{\cC_{g,[n]},\bar\xi}$ the continuous
relative completions of
$\pi_1(\M_{g,[n]/\bar k},\bar\xi)$ and
$\pi_1(\cC_{g,[n]/\bar k},\bar x)$ with respect to
$\rho_{\Q_\ell,[\bar\xi]}$ and $\rho_{\cC_{\Q_\ell},[\bar\xi]}$, respectively.
Let $\U_{g,[n],\bar\xi}$ and $\U_{\cC_{g,[n]},\bar\xi}$ be their unipotent
radicals, and write
\[
\g_{g,[n],\bar\xi},\ \g_{\cC_{g,[n]},\bar\xi},\ 
\u_{g,[n],\bar\xi},\ \u_{\cC_{g,[n]},\bar\xi}
\]
for the Lie algebras of
$\cG_{g,[n],\bar\xi}$, $\cG_{\cC_{g,[n]},\bar\xi}$,
$\U_{g,[n],\bar\xi}$, and $\U_{\cC_{g,[n]},\bar\xi}$, respectively.

\textcolor{black}{Unlike the ordered case, we do not know whether the continuous relative completions
in characteristic $0$ and characteristic $p$ admit a direct comparison in the unordered setting.}
Nevertheless, the following result provides the structural input needed for our later arguments and may be viewed as the characteristic $p$ analogue of \cite[Lem.~5.22]{LW25}.

\begin{lemma}\label{lem:unordered-h1-Sp-invariants}
For $g\geq 2$ and $n\geq 1$, there is an isomorphism
$$
\Hom_{\Sp}^\cts(H_1(\u_{g,[n],\bar\xi}), H_{\Q_\ell}) \cong \Q_\ell.
$$
\end{lemma}

\begin{proof}
By the cohomological description of relative completion
\cite[Thm.~3.8]{hain_relwtfilt}, there is a natural isomorphism
\[
\Hom_{\Sp}^{\cts}\!\bigl(H_1(\u_{g,[n], \bar\xi}), H_{\Q_\ell}\bigr)
\;\cong\;
H^1_{\cts}\!\bigl(\pi_1(\M_{g,[n]/\bar k}), H_{\Q_\ell}\bigr).
\]
Since the monodromy representation
$\pi_1(\M_{g,n/\bar k})\to\Sp(H_{\Q_\ell})$ is Zariski dense, we have
$H_{\Q_\ell}^{\pi_1(\M_{g,n/\bar k})}=0$. \textcolor{black}{Applying the Hochschild--Serre
five-term exact sequence} associated with
$1\to\pi_1(\M_{g,n/\bar k})\to\pi_1(\M_{g,[n]/\bar k})\to S_n\to1$
\textcolor{black}{therefore yields}
\[
H^1_{\cts}\!\bigl(\pi_1(\M_{g,[n]/\bar k}), H_{\Q_\ell}\bigr)
\cong
H^1_{\cts}\!\bigl(\pi_1(\M_{g,n/\bar k}), H_{\Q_\ell}\bigr)^{S_n}.
\]

Again using~\cite[Thm.~3.8]{hain_relwtfilt}, we have natural isomorphisms
\begin{align*}
H^1_{\cts}\!\bigl(\pi_1(\M_{g,n/\bar k}), H_{\Q_\ell}\bigr)
&\cong \Hom_{\Sp}^{\cts}\!\bigl(H_1(\u_{g,n,\bar\xi}), H_{\Q_\ell}\bigr) \\
&\cong \Hom_{\Sp}^{\cts}\!\bigl(H_1(\u_{g,n,\bar\eta}), H_{\Q_\ell,\bar\eta}\bigr) \\
&\cong H^1_{\cts}\!\bigl(\pi_1(\M_{g,n/\bar K},\bar\eta), H_{\Q_\ell,\bar\eta}\bigr) \\
&\cong \bigoplus_{j=1}^n \Q_\ell ,
\end{align*}
where the second isomorphism is induced by
Proposition~\ref{prop:basechange-rel-completions}, and the final isomorphism is
induced by the characteristic classes associated with the tautological sections
of the universal curve (cf.~\cite[\S5.6--\S5.7]{LW25}).

\textcolor{black}{The symmetric group $S_n$ acts by permuting the summands of}
$\bigoplus_{j=1}^n \Q_\ell$. \textcolor{black}{Taking invariants therefore yields}
\[
H^1_{\cts}\!\bigl(\pi_1(\M_{g,[n]/\bar k}), H_{\Q_\ell}\bigr)
\;\cong\;
\Q_\ell,
\]
which proves the result.
\end{proof}

\subsection{Key exact sequences in relative completion}
We now introduce the exact sequences in continuous relative completion that play
a key role in the proof of the main results.

Denote by $\p^{\an}_{g,n}$ the Lie algebra of the $\ell$-adic unipotent completion
of $P_{g,n}$. In the case of $n=1$, denote $\p^\an_{g,1}$ by $\p^\an_{g}$. 
\subsubsection{Key exact sequences for the universal curve}
Since $F^\un_g$ has trivial center, applying continuous relative completion to the diagram of
Proposition~\ref{prop:ordered-unordered-homotopy-rel-prol}, we obtain a
commutative diagram of pro-algebraic $\Q_\ell$-groups with exact rows
(see \cite[Prop.~7.6]{wat_rational_points_inp})
\begin{equation} \label{eq:ordered-unordered-extensions}
\vcenter{\xymatrix@C=15pt@R=15pt{
1\ar[r]&F^\un_{g}\ar[r]\ar@{=}[d]&\cG_{\cC_{g,n}, \bar\xi}\ar[r]\ar[d]&\cG_{g,n,\bar\xi}\ar[r]\ar[d]&1\\
1\ar[r]&F^\un_{g}\ar[r]&\cG_{\cC_{g,[n]}, \bar\xi}\ar[r]&\cG_{g,[n],\bar\xi}\ar[r]&1
}}
\end{equation}

Since all morphisms in~\eqref{eq:ordered-unordered-extensions} land in the
unipotent radicals, this diagram restricts to a commutative diagram of
prounipotent $\Q_\ell$-groups with exact rows:
\[
\xymatrix@C=10pt@R=10pt{
1 \ar[r] & F^\un_{g} \ar[r]\ar@{=}[d] &
\U_{\cC_{g,n}, \bar\xi} \ar[r]\ar[d] &
\U_{g,n,\bar\xi} \ar[r]\ar[d] & 1 \\
1 \ar[r] & F^\un_{g} \ar[r] &
\U_{\cC_{g,[n]}, \bar\xi} \ar[r] &
\U_{g,[n],\bar\xi} \ar[r] & 1 .
}
\]

Passing to Lie algebras, we obtain a commutative diagram of pronilpotent Lie
algebras with exact rows:
\[
\xymatrix@C=10pt@R=10pt{
0\ar[r]&\p_g\ar[r]\ar@{=}[d]&\u_{\cC_{g,n}, \bar\xi}\ar[r]\ar[d]&\u_{g,n,\bar\xi}\ar[r]\ar[d]&0\\
0\ar[r]&\p_g\ar[r]&\u_{\cC_{g,[n]}, \bar\xi}\ar[r]&\u_{g,[n],\bar\xi}\ar[r]&0.
}
\]
Applying the functor $H_1$ to the Lie algebra diagram above, we obtain the
following comparison between the ordered and unordered cases.
\begin{lemma} \label{lem:homology-fiber-comparison}
For $g\geq2$, $n\geq 1$, and $\ell\neq p$ an odd prime, there is a commutative diagram
of $\Sp(H_\Ql)$-representations with exact rows:
\[
\xymatrix@C=15pt@R=15pt{
0 \ar[r] & H_1(\p_{g}) \ar[r]\ar@{=}[d] &
H_1(\u_{\cC_{g,n},\bar\xi}) \ar[r]\ar[d] &
H_1(\u_{g,n,\bar\xi}) \ar[r]\ar[d] & 0 \\
0 \ar[r] & H_1(\p_{g}) \ar[r] &
H_1(\u_{\cC_{g,[n]},\bar\xi}) \ar[r] &
H_1(\u_{g,[n],\bar\xi}) \ar[r] & 0.
}
\]
\end{lemma}

\begin{proof}
The exactness of the top row
\[
0 \longrightarrow H_1(\mathfrak{p}_g) \longrightarrow H_1(\mathfrak{u}_{\mathcal{C}_{g,n},\bar{\xi}}) \longrightarrow H_1(\mathfrak{u}_{g,n,\bar{\xi}}) \longrightarrow 0
\]
follows from the comparison with the characteristic $0$ case, using Proposition \ref{prop:basechange-rel-completions} and \cite[Cor.~5.16]{LW25}.

For \textcolor{black}{the bottom row,} the unordered case, right exactness of the functor $H_1$ implies exactness at the right; it remains only to show that the map $H_1(\mathfrak{p}_g) \to H_1(\mathfrak{u}_{\mathcal{C}_{g, [n]},\bar{\xi}})$ is injective. To this end, we consider the natural weight filtration $W_\bullet\mathfrak{p}_g$ arising from the weighted completion of $\pi_1(\mathcal{C}_{g,n/\mathbb{F}_p})$ (see \cite{wat_rational_points_inp}), a variant of relative completion developed by Hain and Matsumoto (see \cite{hain-matsumoto, hain_rational}).

A key point is that $H_1(\mathfrak{p}_g)$ is pure of weight $-1$. Consequently, the weight filtration on $\mathfrak{p}_g$ coincides with its lower central series. Since the Hodge weight filtration on $\mathfrak{p}^{\mathrm{an}}_g$ is also known to coincide with its lower central series, the comparison isomorphism $\mathfrak{p}^\an_g \cong \mathfrak{p}_g$, which is induced by the isomorphism $\tilde{\lambda}: \pi_1(C_{\bar{\eta}})^{\mathrm{pro}\,\ell} \xrightarrow{\sim} \pi_1(C_{\bar{\xi}})^{\mathrm{pro}\,\ell}$ from Proposition \ref{prop:specialization-rel-prol}, is strictly compatible with weights. This implies that the induced isomorphism on derivation algebras is an isomorphism of \textit{filtered} Lie algebras:
\[
(\mathrm{Der}(\mathfrak{p}_g), W_\bullet) \cong (\mathrm{Der}(\mathfrak{p}^{\mathrm{an}}_g), W_\bullet).
\]

Consider the commutative diagram induced by $\tilde{\lambda}$:
\[
\xymatrix@C=15pt@R=15pt{
\mathfrak{p}_g \ar[r] & \mathrm{Der}(\mathfrak{p}_g) \\
\mathfrak{p}^{\mathrm{an}}_g \ar[r] \ar[u]^\cong & \mathrm{Der}(\mathfrak{p}^{\mathrm{an}}_g) \ar[u]^\cong,
}
\]
which yields the induced diagram on the weight $-1$ graded pieces:
\[
\xymatrix@C=15pt@R=15pt{
H_1(\mathfrak{p}_g) \ar[r] & \mathrm{Gr}^W_{-1}\mathrm{Der}(\mathfrak{p}_g) \\
H_1(\mathfrak{p}^{\mathrm{an}}_g) \ar[r] \ar[u]^\cong & \mathrm{Gr}^W_{-1}\mathrm{Der}(\mathfrak{p}^{\mathrm{an}}_g) \ar[u]^\cong.
}
\]
Since the functor $\mathrm{Gr}^W$ is exact on the category of MHS, the map $\mathrm{Gr}^W_{-1}\mathfrak{p}^{\mathrm{an}}_g = H_1(\mathfrak{p}^{\mathrm{an}}_g) \to \mathrm{Gr}^W_{-1}\mathrm{Der}(\mathfrak{p}^{\mathrm{an}}_g)$ is injective. Consequently, the top map $H_1(\mathfrak{p}_g) \to \mathrm{Gr}^W_{-1}\mathrm{Der}(\mathfrak{p}_g)$ is also injective.

In the case over $\bar\Q$, the Lie algebra $\mathfrak{u}_{\mathcal{C}_{g,n}}$ is negatively weighted, i.e., $\mathfrak{u}_{\mathcal{C}_{g,n}} = W_{-1}\mathfrak{u}_{\mathcal{C}_{g,n}}$. Therefore, the conjugation action of $\mathfrak{u}_{\mathcal{C}_{g,n}}$ on $\mathfrak{p}^{\mathrm{an}}_g$ lands in $W_{-1}\mathrm{Der}(\mathfrak{p}^{\mathrm{an}}_g)$. The comparison isomorphism from Proposition \ref{prop:specialization-rel-prol} induces the commutative diagram:
\[
\xymatrix@C=15pt@R=15pt{
\mathfrak{p}_g \ar[r] & \mathfrak{u}_{\mathcal{C}_{g,n},\bar{\xi}}\ar[r]^-{\mathrm{ad}} & \mathrm{Der}(\mathfrak{p}_g) \\
\mathfrak{p}^{\mathrm{an}}_g \ar[r] \ar[u]^\cong & \mathfrak{u}_{\mathcal{C}_{g,n}}\ar[r]^-{\mathrm{ad}} \ar[u]^\cong & \mathrm{Der}(\mathfrak{p}^{\mathrm{an}}_g) \ar[u]^\cong.
}
\]
It follows that the image of $\mathfrak{u}_{\mathcal{C}_{g, n},\bar{\xi}}$ lies in $W_{-1}\mathrm{Der}(\mathfrak{p}_g)$. Furthermore, since the map $\mathfrak{u}_{\mathcal{C}_{g, n},\bar{\xi}} \to \mathfrak{u}_{\mathcal{C}_{g, [n]},\bar{\xi}}$ is surjective and the adjoint action factors through this quotient, we obtain a sequence of maps:
\[
\mathfrak{p}_g \longrightarrow \mathfrak{u}_{\mathcal{C}_{g, n},\bar{\xi}} \longrightarrow \mathfrak{u}_{\mathcal{C}_{g, [n]},\bar{\xi}} \longrightarrow W_{-1}\mathrm{Der}(\mathfrak{p}_g).
\]
Since the target space $\mathrm{Gr}^W_{-1}\mathrm{Der}(\mathfrak{p}_g)$ is abelian, the composition induces the sequence on homology:
\[
H_1(\mathfrak{p}_g) \longrightarrow H_1(\mathfrak{u}_{\mathcal{C}_{g, n},\bar{\xi}}) \longrightarrow H_1(\mathfrak{u}_{\mathcal{C}_{g, [n]},\bar{\xi}}) \longrightarrow \mathrm{Gr}^W_{-1}\mathrm{Der}(\mathfrak{p}_g).
\]
Because the composite map from $H_1(\mathfrak{p}_g)$ to $\mathrm{Gr}^W_{-1}\mathrm{Der}(\mathfrak{p}_g)$ was shown to be injective, the map $H_1(\mathfrak{p}_g) \to H_1(\mathfrak{u}_{\mathcal{C}_{g, [n]},\bar{\xi}})$ must also be injective.
\end{proof}

\subsubsection{Key exact sequences for the family $\mathcal M_{g,n}\to\mathcal M_g$}
Next, since $\p^{\an}_{g,n}$ has trivial center \cite[p.~201]{NTU} and
$\p^{\an}_{g,n}\cong \p_{g,n}$, the same holds for $\p_{g,n}$ and for
$F^{\un}_{g,n}$.

By right exactness and naturality of relative completion, applying continuous
relative completion to the diagram of
Proposition~\ref{prop:ordered-unordered-rel-prol} yields a commutative diagram
of pro-algebraic $\Q_\ell$-groups with exact rows
\[
\xymatrix@C=10pt@R=10pt{
1 \ar[r] & F^{\un}_{g,n} \ar[r]\ar[d] &
\cG_{g,n,\bar\xi} \ar[r]\ar[d] &
\cG_{g,\bar\xi} \ar[r]\ar@{=}[d] & 1 \\
& F^{\un}_{g,[n]} \ar[r] &
\cG_{g,[n],\bar\xi} \ar[r] &
\cG_{g,\bar\xi} \ar[r] & 1 .
}
\]

Since all morphisms land in the unipotent radicals, this restricts to a
commutative diagram of prounipotent $\Q_\ell$-groups with exact rows
\[
\xymatrix@C=10pt@R=10pt{
1 \ar[r] & F^{\un}_{g,n} \ar[r]\ar[d] &
\U_{g,n,\bar\xi} \ar[r]\ar[d] &
\U_{g,\bar\xi} \ar[r]\ar@{=}[d] & 1 \\
& F^{\un}_{g,[n]} \ar[r] &
\U_{g,[n],\bar\xi} \ar[r] &
\U_{g,\bar\xi} \ar[r] & 1 .
}
\]

Passing to Lie algebras, we obtain a commutative diagram of pronilpotent Lie
algebras with exact rows
\[
\xymatrix@C=10pt@R=10pt{
0 \ar[r] & \p_{g,n} \ar[r]\ar[d] &
\u_{g,n,\bar\xi} \ar[r]\ar[d] &
\u_{g,\bar\xi} \ar[r]\ar@{=}[d] & 0 \\
& \p_{g,[n]} \ar[r] &
\u_{g,[n],\bar\xi} \ar[r] &
\u_{g,\bar\xi} \ar[r] & 0 .
}
\]
Applying the functor $H_1$ to the preceding diagram of pronilpotent Lie algebras
yields the following commutative diagram of $\Sp(H_{\Q_\ell})$-representations.
\begin{lemma}\label{lem:ordered-unordered-comparison}
For $g\geq2$, $n\geq 1$, and $\ell\not =p$ an odd prime, there is a commutative
diagram of $\Sp(H_\Ql)$-representations with exact rows
\[
\xymatrix@C=10pt@R=10pt{
0 \ar[r] & H_1(\p_{g,n}) \ar[r]\ar[d] &
H_1(\u_{g,n,\bar\xi}) \ar[r]\ar[d] &
H_1(\u_{g,\bar\xi}) \ar[r]\ar@{=}[d] & 0 \\
0\ar[r]& H_1(\p_{g,[n]}) \ar[r] &
H_1(\u_{g,[n],\bar\xi}) \ar[r] &
H_1(\u_{g,\bar\xi}) \ar[r] & 0 .
}
\]
\end{lemma}

\begin{proof}
Consider the extension 
\[
1 \longrightarrow \mathcal{U}_{g,n, \bar{\xi}} \longrightarrow \mathcal{G}_{g,n,\bar{\xi}} \longrightarrow \mathrm{Sp}(H_{\mathbb{Q}_\ell}) \longrightarrow 1.
\]
The group $\mathcal{G}_{g,n,\bar{\xi}}$ acts by conjugation on the ordered fiber $F^\mathrm{un}_{g,n}$ and the unipotent radical $\mathcal{U}_{g,n,\bar{\xi}}$. These actions induce $\mathcal{G}_{g,n,\bar{\xi}}$-equivariant maps on the homology of the corresponding Lie algebras:
\[
H_1(\mathfrak{p}_{g,n}) \longrightarrow H_1(\mathfrak{u}_{g,n,\bar{\xi}}) \quad \text{and} \quad H_1(\mathfrak{u}_{g,n,\bar{\xi}}) \longrightarrow H_1(\mathfrak{u}_{g, \bar{\xi}}).
\]
Since the action of $\mathcal{G}_{g,n,\bar{\xi}}$ on these abelianizations factors through $\mathrm{Sp}(H_{\mathbb{Q}_\ell})$, these maps are $\mathrm{Sp}(H_{\mathbb{Q}_\ell})$-equivariant.

For the unordered case, similarly, the conjugation action of $\mathcal{G}_{g,[n],\bar{\xi}}$ on the unipotent radicals $\mathcal{U}_{g,[n],\bar{\xi}}$ and $\mathcal{U}_{g,\bar{\xi}}$ induces a map on their Lie algebra homology:
\[
H_1(\mathfrak{u}_{g,[n],\bar{\xi}}) \longrightarrow H_1(\mathfrak{u}_{g,\bar{\xi}}).
\]
This map is $\mathcal{G}_{g,[n],\bar{\xi}}$-equivariant. Since the action of $\mathcal{G}_{g,[n],\bar{\xi}}$ on the abelianization of its unipotent radical factors through $\mathrm{Sp}(H_{\mathbb{Q}_\ell})$, this induced map is $\mathrm{Sp}(H_{\mathbb{Q}_\ell})$-equivariant. Next, recall from Proposition \ref{prop:H1-pgn-quot} that we have the isomorphism
\[
H_1(F^\mathrm{un}_{g, [n]}) \cong \left( H_1(F^\mathrm{un}_{g,n}) \right)_{S_n}.
\]
The conjugation action of $\mathcal{G}_{g,n, \bar{\xi}}$ on $H_1(F^\mathrm{un}_{g,n}) \cong H_{\mathbb{Q}_\ell}^{\oplus n}$ commutes with the permutation of the marked points (acting diagonally on the summands), so it descends to a well-defined action on $H_1(\mathfrak{p}_{g, [n]}) \cong H_1(F^\mathrm{un}_{g,[n]})$. Consequently, the natural morphism
\[
H_1(\mathfrak{p}_{g, [n]}) \longrightarrow H_1(\mathfrak{u}_{g, [n], \bar{\xi}})
\]
is $\mathcal{G}_{g,n,\bar{\xi}}$-equivariant. Since the induced action on both the source (the coinvariants) and the target factors through the symplectic representation, this morphism is $\mathrm{Sp}(H_{\mathbb{Q}_\ell})$-equivariant.

For the exactness of the top row, applying the continuous relative completion functor to the exact sequence of groups in Proposition \ref{prop:ordered-specialization-rel-prol} yields the commutative diagram:
\[
\xymatrix@C=15pt@R=15pt{
& H_1(\mathfrak{p}_{g,n}) \ar[r] & H_1(\mathfrak{u}_{g,n,\bar{\xi}}) \ar[r] & H_1(\mathfrak{u}_{g,\bar{\xi}}) \ar[r] & 0 \\
0 \ar[r] & H_1(\mathfrak{p}^{\mathrm{an}}_{g,n}) \ar[r]\ar[u]_\cong & H_1(\mathfrak{u}_{g,n}) \ar[r]\ar[u]_\cong & H_1(\mathfrak{u}_g) \ar[r]\ar[u]_\cong & 0.
}
\]
Here, the vertical isomorphisms are induced by the comparison isomorphisms provided by Proposition \ref{prop:ordered-specialization-rel-prol} and Proposition \ref{prop:basechange-rel-completions}. A sketch of the argument for the exactness of the bottom row at the left is as follows. The Lie algebras $\mathfrak{p}^{\mathrm{an}}_{g,n}$, $\mathfrak{u}_{g,n}$, and $\mathfrak{u}_g$ carry natural weight filtrations $W$ induced by their underlying Mixed Hodge Structures \cite[Sec.~3]{hain_infpretor}. Since the first homology $H_1$ is pure of weight $-1$, these weight filtrations coincide with the lower central series \cite[Lem.~4.7]{hain_infpretor}. The map $\mathfrak{p}^{\mathrm{an}}_{g,n} \to \mathfrak{u}_{g,n}$ is injective because $\mathfrak{p}^{\mathrm{an}}_{g,n}$ is center-free. Since the functor $\mathrm{Gr}^W_\bullet$ is exact on the category of MHS, this injectivity is preserved on the graded pieces. In particular, the induced map on the weight $-1$ component,
\[
\mathrm{Gr}^W_{-1}\mathfrak{p}^{\mathrm{an}}_{g,n} = H_1(\mathfrak{p}^{\mathrm{an}}_{g,n}) \longrightarrow H_1(\mathfrak{u}_{g,n}) = \mathrm{Gr}^W_{-1}\mathfrak{u}_{g,n},
\]
is injective.

Since the bottom row is exact (as established above) and the vertical maps are isomorphisms, the top row is exact. Therefore, the map $H_1(\mathfrak{p}_{g,n})\to H_1(\mathfrak{u}_{g,n,\bar{\xi}})$ is also injective.

Applying the functor $H_1$ to the exact sequence
\[
\mathfrak{p}_{g, [n]} \longrightarrow \mathfrak{u}_{g, [n],\bar{\xi}} \longrightarrow \mathfrak{u}_{g, \bar{\xi}} \longrightarrow 0,
\]
we obtain the exact sequence of $\mathrm{Sp}(H_{\mathbb{Q}_\ell})$-representations:
\[
H_1(\mathfrak{p}_{g, [n]}) \xrightarrow{j} H_1(\mathfrak{u}_{g, [n],\bar{\xi}}) \xrightarrow{\pi} H_1(\mathfrak{u}_{g,\bar{\xi}}) \longrightarrow 0,
\]
where we denote the projection by $\pi$ and the natural $\mathrm{Sp}(H_{\mathbb{Q}_\ell})$-equivariant map by $j$. By exactness, we have $\ker \pi = \mathrm{im}\, j$.

Recall from Proposition \ref{prop:H1-pgn-quot} that $H_1(\mathfrak{p}_{g,[n]}) \cong H_{\mathbb{Q}_\ell}$. Since the standard representation $H_{\mathbb{Q}_\ell}$ is irreducible, the map $j$ must be either the zero map or injective.

Suppose for the sake of contradiction that $j$ is the zero map. Then $\mathrm{im}\, j = 0$, which implies $\ker \pi = 0$. Consequently, $\pi$ would be an isomorphism. However, this implies that $\mathrm{Hom}_{\mathrm{Sp}}^\mathrm{cts}(H_1(\mathfrak{u}_{g,[n],\bar{\xi}}), H_{\mathbb{Q}_\ell}) = 0$ by Proposition \ref{prop:basechange-rel-completions} and \cite[Thm.~9.11]{hain_rational}. This contradicts Lemma \ref{lem:unordered-h1-Sp-invariants}. Thus, $j$ must be an injection. This completes the proof.
\end{proof}

\begin{remark}\label{Sn decomp for ugn}
Consider the exact sequence for the ordered case (the top row of Lemma \ref{lem:ordered-unordered-comparison}):
\[
0 \longrightarrow H_1(\mathfrak{p}_{g,n}) \longrightarrow H_1(\mathfrak{u}_{g,n,\bar{\xi}}) \longrightarrow H_1(\mathfrak{u}_{g,\bar{\xi}}) \longrightarrow 0.
\]
This is a short exact sequence of $\mathrm{Sp}(H_{\mathbb{Q}_\ell})$-representations. Since the symplectic group $\mathrm{Sp}(H_{\mathbb{Q}_\ell})$ is reductive, its category of finite-dimensional representations is semisimple. Consequently, every short exact sequence splits, providing an $\mathrm{Sp}(H_{\mathbb{Q}_\ell})$-equivariant isomorphism:
\[
H_1(\mathfrak{u}_{g,n,\bar{\xi}}) \cong H_1(\mathfrak{p}_{g,n}) \oplus H_1(\mathfrak{u}_{g,\bar{\xi}})\cong H_\Ql^{\oplus n}\oplus H_1(\u_{g,\bar\xi}).
\]
This splitting descends to the unordered case (the bottom row of Lemma \ref{lem:ordered-unordered-comparison}). The vertical maps in the diagram correspond to taking coinvariants under the action of the symmetric group $S_n$. Since the functor of taking coinvariants commutes with direct sums, and $S_n$ acts trivially on the term $H_1(\mathfrak{u}_{g,\bar{\xi}})$, we obtain:
\[
H_1(\mathfrak{u}_{g,[n],\bar{\xi}}) \cong \left( H_1(\mathfrak{p}_{g,n}) \right)_{S_n} \oplus \left( H_1(\mathfrak{u}_{g,\bar{\xi}}) \right)_{S_n} \cong H_1(\mathfrak{p}_{g,[n]}) \oplus H_1(\mathfrak{u}_{g,\bar{\xi}}) \cong H_\Ql\oplus H_1(\mathfrak{u}_{g,\bar{\xi}}).
\]
This provides an $\mathrm{Sp}(H_{\mathbb{Q}_\ell})$-decomposition for the unordered relative completion that is compatible with the $S_n$-action.
\end{remark}

\subsection{Proof of the main results}

Finally, we prove Theorem~\ref{thm:non-split-positive-char} and
Corollary~\ref{cor:no-section}.

\begin{proof}[Proof of Theorem~\ref{thm:non-split-positive-char}]
Suppose that the canonical projection
$\tilde{\pi}_{g, [n]} :
\mathcal{G}_{\mathcal{C}_{g, [n]}, \bar{\xi}}
\to
\mathcal{G}_{g, [n],\bar{\xi}}$
admits a section $s$.
\textcolor{black}{Since the diagram~\eqref{eq:ordered-unordered-extensions} is a
pullback diagram of pro-algebraic $\Q_\ell$-groups,}
the section $s$ induces a section $t$ of the projection
$\mathcal{G}_{\mathcal{C}_{g,n},\bar{\xi}}
\to
\mathcal{G}_{g,n,\bar{\xi}}$
\textcolor{black}{such that the following diagram commutes:}
\begin{equation*}
\vcenter{\xymatrix@C=15pt@R=15pt{
1\ar[r]&F^\un_{g}\ar[r]\ar@{=}[d]
&\cG_{\cC_{g,n}, \bar\xi}\ar[r]\ar[d]
&\cG_{g,n,\bar\xi}\ar[r]\ar[d]\ar@/^1pc/[l]_{t}&1\\
1\ar[r]&F^\un_{g}\ar[r]
&\cG_{\cC_{g,[n]}, \bar\xi}\ar[r]
&\cG_{g,[n],\bar\xi}\ar[r]\ar@/^1pc/[l]^{s}&1.
}}
\end{equation*}

\textcolor{black}{Passing to Lie algebras,} this yields a corresponding commutative
diagram of pronilpotent Lie algebras:
\[
\xymatrix@C=15pt@R=15pt{
0\ar[r]&\p_g\ar[r]\ar@{=}[d]
&\u_{\cC_{g,n}, \bar\xi}\ar[r]\ar[d]
&\u_{g,n,\bar\xi}\ar[r]\ar[d]\ar@/^1pc/[l]^{dt}&0\\
0\ar[r]&\p_g\ar[r]
&\u_{\cC_{g,[n]}, \bar\xi}\ar[r]
&\u_{g,[n],\bar\xi}\ar[r]\ar@/^1pc/[l]^{ds}&0.
}
\]

Applying Lemma~\ref{lem:homology-fiber-comparison},
Lemma~\ref{lem:ordered-unordered-comparison}, and
Remark~\ref{Sn decomp for ugn},
we obtain the commutative diagram
\begin{equation} \label{eq:abelianized-splitting}
\vcenter{\xymatrix@C=15pt@R=15pt{
0\ar[r]&H_\Ql\ar[r]\ar@{=}[d]
&H_\Ql\oplus H_1(\u_{g,\bar\xi})\oplus H_\Ql^{\oplus n}\ar[r]\ar[d]
&H_1(\u_{g,\bar\xi})\oplus H_\Ql^{\oplus n}\ar[r]\ar[d]\ar@/^1pc/[l]^{dt^\ab}&0\\
0\ar[r]&H_\Ql\ar[r]
&H_\Ql\oplus H_1(\u_{g,\bar\xi})\oplus H_\Ql\ar[r]
&H_1(\u_{g,\bar\xi})\oplus H_\Ql\ar[r]\ar@/^1pc/[l]^{ds^\ab}&0,
}}
\end{equation}
where $dt^{\mathrm{ab}}$ and $ds^{\mathrm{ab}}$
denote the abelianizations of $dt$ and $ds$, respectively.

By the comparison isomorphisms established in
Proposition~\ref{prop:basechange-rel-completions},
the section $t$ induces a section
$t_{\bar{\Q}}$ of
$\mathcal{G}_{\mathcal{C}_{g,n}} \to \mathcal{G}_{g,n}$,
and consequently an
$\Sp(H_{\Q_\ell})$-equivariant linear section
$dt^{\mathrm{ab}}_{\bar{\Q}}$ of
$H_1(\mathfrak{u}_{\mathcal{C}_{g,n}})
\to
H_1(\mathfrak{u}_{g,n})$.
According to~\cite[Prop.~6.4]{LW25}, for $g \ge 3$,
there are exactly $n$ such sections arising from a section of
$\mathcal{G}_{\mathcal{C}_{g,n}} \to \mathcal{G}_{g,n}$.
More precisely,
\textcolor{black}{after identifying via the comparison isomorphism,}
the restriction of $dt^{\mathrm{ab}}$ to the summand
$H_{\Q_\ell}^{\oplus n}$
maps into the component
$H_1(\mathfrak{p}_g) \cong H_{\Q_\ell}$
(corresponding to the fiber $C_{\bar{\xi}}$)
by the formula
\[
dt^{\mathrm{ab}} : (x_1,\ldots, x_n) \longmapsto x_i
\]
for some fixed $i \in \{1, \ldots, n\}$.

The remainder of the argument follows the proof of
\cite[Thm.~1]{LW25}.
It is shown there that such a map $dt^{\mathrm{ab}}$
(projection to a single coordinate)
is not invariant under the action of $S_n$ for $n \ge 2$.
This contradicts the commutativity of
diagram~\eqref{eq:abelianized-splitting},
which requires the section to descend to the
$S_n$-coinvariants.
Therefore, no such section $s$ can exist.
\end{proof}

\begin{proof}[Proof of Corollary~\ref{cor:no-section}]
Every continuous section of the projection
$\pi_1(\mathcal{C}_{g, [n]/k})
\to
\pi_1(\mathcal{M}_{g, [n]/k})$
induces a section of the geometric fundamental groups
$\pi_1(\mathcal{C}_{g, [n]/\bar{k}})
\to
\pi_1(\mathcal{M}_{g, [n]/\bar{k}})$.
By the universal property of continuous relative completion,
this in turn induces a section of the relative completions
$\mathcal{G}_{\mathcal{C}_{g,[n]},\bar{\xi}}
\to
\mathcal{G}_{g, [n], \bar{\xi}}$,
and consequently an
$\Sp(H_{\Q_\ell})$-equivariant linear section of
\[
H_1(\mathfrak{u}_{\mathcal{C}_{g, [n]}, \bar{\xi}})
\longrightarrow
H_1(\mathfrak{u}_{g, [n], \bar{\xi}}).
\]
By Theorem~\ref{thm:non-split-positive-char},
no such section exists.
Thus, the original projection admits no section.
\end{proof}
\section{Notation and Glossary}

Throughout this paper, \(k\) denotes a field of characteristic \(p>0\), and
\(\bar{k}\) its algebraic closure.  
Let \(R := W(\bar{k})\) be the ring of Witt vectors of \(\bar{k}\), and let
\(K\) denote the fraction field of \(R\).
We write \(\bar{\xi}\) for a geometric closed point of \(\Spec R\) and
\(\bar{\eta}\) for a geometric generic point of \(\Spec R\).

Let \(L\) be an algebraically closed field of characteristic \(p \ge 0\).
For a stack \(\chi\) over \(L\), we denote by \(\bar{\zeta}\) a geometric point
of \(\chi\).

All fundamental groups are \'etale fundamental groups, unless stated otherwise,
for example \(\pi_1^{\top}\) or \(\pi_1^{\orb}\).
All pro-algebraic groups and Lie algebras are defined over \(\Q_\ell\).

\subsection*{Moduli stacks and universal curves}

\begin{itemize}\setlength\itemsep{2pt}
\item $\M_{g,n/k}$: the moduli stack of smooth proper genus-$g$ curves with
      $n$ ordered distinct marked points over $k$.

\item $\M_{g,[n]/k} := [\M_{g,n/k}/S_n]$: the moduli stack of smooth proper
      genus-$g$ curves with $n$ \emph{unordered} marked points.

\item $\pi_{g,n/k} : \cC_{g,n/k} \to \M_{g,n/k}$: the universal curve in the
      ordered case.

\item $\pi_{g,[n]/k} : \cC_{g,[n]/k} \to \M_{g,[n]/k}$: the universal curve in the
      unordered case.
\end{itemize}

\subsection*{Configuration spaces and their fundamental groups}

Let $C$ be a smooth proper connected curve of genus $g$ over $\bar k$.

\begin{itemize}\setlength\itemsep{2pt}
\item $F(C,n) := C^n \setminus \Delta$: the ordered configuration space of $n$
      distinct points on $C$.

\item $F(C,[n]) := F(C,n)/S_n$: the unordered configuration space.

\item $F_{g,n} := \pi_1(F(C,n))$: the algebraic fundamental group of the ordered
      configuration space.

\item $F_{g,[n]} := \pi_1(F(C,[n]))$: the algebraic fundamental group of the
      unordered configuration space.
\end{itemize}

\subsection*{$\ell$-adic unipotent completions}

\begin{itemize}\setlength\itemsep{2pt}
\item $F_{g,n}^{\un}$, $F_{g,[n]}^{\un}$:
      the $\ell$-adic unipotent completions of
      $F_{g,n}$ and $F_{g,[n]}$, respectively.

\item $\p_{g,n} := \Lie(F_{g,n}^{\un})$, \quad
      $\p_{g,[n]} := \Lie(F_{g,[n]}^{\un})$:
      the associated pronilpotent Lie algebras.

\item $\p_g := \p_{g,1}$:
      the Lie algebra associated with the unipotent completion of
      $\pi_1(C)$.
\end{itemize}

\subsection*{Continuous relative completions}

Let
\[
\rho_{\Q_\ell} :
\pi_1(\M_{g,n/\bar k},\bar\xi)
\longrightarrow
\Sp(H_{\Q_\ell})
\]
denote the $\ell$-adic monodromy representation, where
\[
H_{\Q_\ell} := H^1_{\et}(C,\Q_\ell(1))
\]
is the standard symplectic $\Q_\ell$-vector space associated with a geometric
fiber of the universal curve.

The same construction applies verbatim over any algebraically closed base field
\[
L \in \{\bar k,\ \bar K,\ \bar{\Q}\},
\]
with respect to a chosen geometric point \(\bar\zeta\) of the corresponding
stack. Unless explicitly indicated otherwise, we use the same notation for the
continuous relative completions over each of these bases, specifying the
geometric point only when it is relevant (e.g.\ in specialization arguments).

\begin{itemize}\setlength\itemsep{2pt}
\item $\cG_{g,n,\bar\zeta}$, $\cG_{g,[n],\bar\zeta}$:
      the continuous relative completions of
      $\pi_1(\M_{g,n/L},\bar\zeta)$ and
      $\pi_1(\M_{g,[n]/L},\bar\zeta)$
      with respect to the induced $\ell$-adic monodromy representation
      $\rho_{\Q_\ell}$.

\item $\cG_{\cC_{g,n},\bar\zeta}$, $\cG_{\cC_{g,[n]},\bar\zeta}$:
      the continuous relative completions of the fundamental groups of the
      ordered and unordered universal curves over \(L\), obtained by composing
      $\rho_{\Q_\ell}$ with the natural homomorphisms induced by the universal
      curve.

\item $\U_{(\cdot)}$:
      the unipotent radical of a continuous relative completion.

\item $\u_{(\cdot)} := \Lie(\U_{(\cdot)})$:
      the associated pronilpotent Lie algebra.
\end{itemize}


\begin{remark}
In the characteristic $0$ setting of our prequel~\cite{LW25}, continuous
relative completions were denoted by $\cG^{\prol}$ in order to distinguish them
from relative completions over~$\Q$.  
In the present paper, all relative completions are continuous and defined over
$\Q_\ell$, and we therefore omit the superscript~$\prol$ to simplify notation.
\end{remark}

\bibliographystyle{alpha}   
\bibliography{references}

@book{SGA1,
  editor    = {A. Grothendieck and M. Raynaud},
  title     = {Rev\^etements \'{E}tales et Groupe Fondamental (SGA 1)},
  series    = {Lecture Notes in Mathematics},
  volume    = {224},
  publisher = {Springer-Verlag},
  address   = {Berlin--New York},
  year      = {1971},
  note      = {S\'eminaire de G\'eom\'etrie Alg\'ebrique du Bois Marie 1960--61.
               MR0354651 (50:7129)}
}

@article{Anderson1974,
  author  = {M. Anderson},
  title   = {Exactness properties of profinite completion functors},
  journal = {Topology},
  volume  = {13},
  year    = {1974},
  pages   = {229--239},
  note    = {MR0354882 (50:7359)}
}

@book{Birman1974,
  author    = {J. Birman},
  title     = {Braids, Links, and Mapping Class Groups},
  series    = {Annals of Mathematics Studies},
  volume    = {82},
  publisher = {Princeton University Press},
  year      = {1974},
  note      = {MR0375281 (51:11477)}
}

@article{FM,
  author  = {W. Fulton and R. MacPherson},
  title   = {A compactification of configuration spaces},
 journal = {Ann. of Math. Stud.},
  volume  = {125},
  year    = {1994},
  pages   = {243--270},
  note    = {MR1311345 (96d:14004)}
}

@article{hain_hodge_rel,
  author  = {R. M. Hain},
  title   = {{The Hodge--de Rham theory of relative Malcev completion}},
  journal = {Ann. Sci. Éc. Norm. Supér. (4)},
  volume  = {31},
  number  = {1},
  year    = {1998},
  pages   = {47--92}
}

@article{hain_completion,
  author  = {R. M. Hain},
  title   = {{Completions of mapping class groups and the cycle $C-C^-$}},
  journal = {Contemp. Math.},
  volume  = {150},
  year    = {1993},
  pages   = {75--105}
}

@incollection{hain_relwtfilt,
  author    = {Hain, R.},
  title     = {Relative weight filtrations on completions of mapping class groups},
  booktitle = {Groups of Diffeomorphisms},
  series    = {Adv. Stud. Pure Math.},
  volume    = {52},
  publisher = {Math. Soc. Japan},
  year      = {2008},
  pages     = {309--368}
}

@article{hain_infpretor,
  author  = {R. M. Hain},
  title   = {Infinitesimal presentations of the {Torelli} groups},
  journal = {J. Amer. Math. Soc. },
  volume  = {10},
  number  = {3},
  year    = {1997},
  pages   = {597--651}
}

@article{HainMatsumoto2009,
  author  = {Hain, R. and Matsumoto, M.},
  title   = {Relative pro-$\ell$ completions of mapping class groups},
 journal = {J. Algebra},
  volume  = {321},
  number  = {11},
  year    = {2009},
  pages   = {3335--3374},
  doi     = {10.1016/j.jalgebra.2009.02.014},
  mrnumber= {2512649}
}

@article{hain-matsumoto,
  author  = {R. Hain and M. Matsumoto},
  title   = {{Weighted completion of Galois groups and Galois actions on the fundamental group of $\mathbb{P}^1 - \{0,1,\infty\}$}},
  journal = {Compos. Math.},
  volume  = {139},
  number  = {2},
  year    = {2003},
  pages   = {119--167}
}

@article{hain_rational,
  author  = {R. Hain},
  title   = {Rational points of universal curves},
  journal = {J. Amer. Math. Soc.},
  volume  = {24},
  number  = {3},
  year    = {2011},
  pages   = {709--769}
}

@article{NTU,
  author  = {H. Nakamura and N. Takao and R. Ueno},
  title   = {Some stability properties of {Teichm{\"u}ller} modular function fields with pro-$\ell$ weight structures},
 journal = {Math. Ann.},
  volume  = {302},
  year    = {1995},
  pages   = {197--213},
  note    = {MR1336334 (96h:14041)}
}

@article{wat_rational_points_inp,
  author  = {T. Watanabe},
  title   = {Rational points of universal curves in positive characteristics},
  journal = {Trans. Amer. Math. Soc.},
  volume  = {372},
  number  = {11},
  year    = {2019},
  pages   = {7639--7676},
  doi     = {10.1090/tran/7842},
  mrnumber= {4029677}
}

@article{LW25,
  author = {M. Luo and T. Watanabe},
  title  = {On the Universal Curves with Unordered Marked Points},
  year   = {2025},
  note   = {Preprint. arXiv:2511.12354}
}

\end{document}